\documentclass{article}

\usepackage[greek, english]{babel}
\usepackage{amsmath}
\usepackage{amssymb}
\usepackage{eurosym}
\usepackage{epsf}
\usepackage{epsfig}
\usepackage{tikz}
\usetikzlibrary{arrows}

\newfont{\gothic}{eufm10}

   \def\C{{\CC}} \def\CC{{\mathbb{C}}}
                   \def\R{{\RR}} 
\def\RR{{\mathbb{R}}}

        \newtheorem{theorem}{Theorem}[section]
\newtheorem{lemma}[theorem]{Lemma}
\newtheorem{proposition}[theorem]{Proposition}
\newtheorem{corollary}[theorem]{Corollary}
\newtheorem{definition1}[theorem]{Definition}
\newenvironment{definition}{\begin{definition1}\rm}{\hfill $\triangle$\end{definition1}}
\newenvironment{proof}{\addvspace\baselineskip\noindent{\it
Proof:}\quad}{\hspace*{\fill}         $\Box$\par\addvspace\baselineskip}
\newenvironment{proofof}[1]{\addvspace\baselineskip\noindent{\it Proof}\quad}{\hspace*{\fill} $\Box$\par\addvspace\baselineskip}

\newtheorem{remark1}[theorem]{Remark}
\newenvironment{remark}{\begin{remark1}\rm}{\hfill $\triangle$\end{remark1}}

\newtheorem{example1}[theorem]{Example}
\newenvironment{example}{\begin{example1}\rm}{\hfill $\triangle$\end{example1}}

\def\barray{\begin{eqnarray*}}             \def\earray{\end{eqnarray*}}
\def\beq{\begin{equation}} \def\eeq{\end{equation}}

\makeatletter \title{Coupled cell networks: \\ semigroups, Lie algebras and normal forms}
\author{Bob Rink\thanks{Department of Mathematics, VU University Amsterdam, The Netherlands, {\tt b.w.rink@vu.nl}.} \ and Jan Sanders\thanks{Department of Mathematics, VU University Amsterdam, The Netherlands, {\tt jan.sanders.a@gmail.com}.}
}
\begin{document}  \hyphenation{boun-da-ry mo-no-dro-my sin-gu-la-ri-ty ma-ni-fold ma-ni-folds re-fe-rence se-cond se-ve-ral dia-go-na-lised con-ti-nuous thres-hold re-sul-ting fi-nite-di-men-sio-nal ap-proxi-ma-tion pro-per-ties ri-go-rous mo-dels mo-no-to-ni-ci-ty pe-ri-o-di-ci-ties mi-ni-mi-zer mi-ni-mi-zers know-ledge ap-proxi-mate pro-per-ty poin-ting ge-ne-ra-li-za-tion}
 
\newcommand{\X}{\mathbb{X}}

\newcommand{\p}{\partial}
\maketitle
\noindent 
\abstract{\noindent We introduce the concept of a semigroup coupled cell network and show that the collection of semigroup network vector fields forms a Lie algebra. This implies that near a dynamical equilibrium the local normal form of a semigroup network is a semigroup network itself. Networks without the semigroup property will support normal forms with a more general network architecture, but these normal forms nevertheless possess the same symmetries and synchronous solutions as the original network. We explain how to compute Lie brackets and normal forms of coupled cell networks and we characterize the SN-decomposition that determines the normal form symmetry. This paper concludes with a generalization to nonhomogeneous networks with the structure of a semigroupoid.
}
\section{Introduction}
Coupled cell networks appear in many of the sciences and range from crystal models and electrical circuits to numerical discretization schemes, Josephson junction arrays, power grids, the world wide web, ecology, neural networks and systems biology. Not surprisingly, there exists an overwhelming amount of literature on coupled cell networks.

The last decade has seen the development of an extensive mathematical theory of dynamical systems with a network structure, cf. \cite{field}, \cite{curious}, \cite{golstew}, \cite{stewartnature}, \cite{pivato}. 
In these network dynamical systems, the evolution of the state of a constituent or ``cell'' is determined by the states of certain particular other cells. It is generally believed that a network structure has an impact on the behavior of a dynamical system, but it is not always clear how and why.
 
 As an example, let us mention a system of differential equations with a
homogeneous coupled cell network structure of the form 
 \begin{align}\label{diffeqnintro}
 \dot x_i =  f(x_{\sigma_1(i)}, \ldots, x_{\sigma_n(i)}) \ \mbox{for} \ 1\leq i\leq N.
\end{align}
These differential equations generate a dynamical system in which the evolution of the variable $x_i$ is only determined by the values of $x_{\sigma_1(i)}, \ldots, x_{\sigma_n(i)}$. The functions
$$\sigma_1, \ldots, \sigma_n: \{1, \ldots, N\} \to \{1, \ldots, N\}$$
can thus be thought of as a network that prescribes how which cells influence which cells.

The literature on network dynamical systems focuses on the analysis of equilibria, periodic solutions, symmetry, synchrony, structural stability and bifurcations. As in the classical theory of dynamical systems, one often faces the task here of computing a local normal form near a dynamical equilibrium.  These normal forms are obtained from coordinate transformations, and in their
computation one calculates Lie brackets of vector fields, either implicitly or explicitly. It is here that one encounters an important technical problem:
\begin{center}
 {\it Differential equations of the form (\ref{diffeqnintro}) in general do not form a Lie-algebra}. 
\end{center}
As a consequence one can not expect that the normal form of a coupled cell network is a coupled cell network as well. This complicates the local analysis and classification of network dynamical systems, because it means that one always has to compute the normal form of a network explicitly to understand its generic behavior - unless one is willing to assume that the network is given in normal form from the beginning, cf. \cite{curious}, \cite{claire}. Normal form computations in \cite{elmhirst}, \cite{krupa}, \cite{claire2}  have revealed that a network structure can have a nontrivial impact on this generic behavior. One wants to understand and predict this. 

In this paper, we will formulate an easily veri\-fiable condition on a network structure under which the coupled cell network vector fields do form a Lie subalgebra of the Lie algebra of vector fields. 
Our main result is the following:
\begin{center}
{\it If $\{\sigma_1, \ldots, \sigma_n\}$ is a semigroup, then the differential equations (\ref{diffeqnintro}) form a Lie algebra.\\ In this case, the local normal form of (\ref{diffeqnintro}) is also of the form (\ref{diffeqnintro}).}
\end{center}
In addition, we show that the Lie bracket of semigroup coupled cell network vector fields can be lifted to a symbolic bracket that only involves the function \(f\). Normal form calculations can be performed at this symbolic level and one only returns to the reality of the differential equation when one is done computing. We also show that the symbolic space carries a dynamics of its own, determined by a certain fundamental network. 

This situation is analogous to that of Hamiltonian vector fields, of which the Lie bracket is determined by the Poisson bracket of Hamiltonian functions. As a consequence, Hamiltonian normal forms are usually computed at the level of functions. Moreover, the symbolic dynamics of Hamiltonian functions is determined by a Poisson structure, cf. \cite{M&R}.

When $\sigma_1, \ldots, \sigma_n$ do not form a semigroup, then we suggest that one simply completes them to the smallest collection
$$\sigma_1,\ldots, \sigma_n, \sigma_{n+1},\ldots, \sigma_{n'}: \{1, \ldots, N\} \to \{1, \ldots, N\}$$
that does form a semigroup under composition. Then (\ref{diffeqnintro}) can be written as
$$\dot x_i = f'(x_{\sigma_1(i)}, \ldots, x_{\sigma_{n'}(i)})\ \mbox{with} \ f'(X_1, \ldots, X_n, X_{n+1}, \ldots, X_{n'}) := f(X_1, \ldots, X_n)\, .$$
The normal form of (\ref{diffeqnintro}) will now lie within the extended class of semigroup coupled cell networks, there being no guarantee that it is again of the original form (\ref{diffeqnintro}). 

Thus one can choose: either to respect any given network structure as if it were a law of nature, so that no normal form can be computed,
or to extend every network to a semigroup network and live with the consequences. 
One can object that even simple networks may need a lot of extension before they form a semigroup. But as an argument in favor of the semigroup approach, let us mention that the symmetries and synchrony spaces of a network are not at all affected by our semigroup extension. This implies in particular that these symmetries and synchrony spaces will also be present in the local
normal form of the network. This latter property is both pleasant and important, if only in view of the large amount of research that has been devoted to symmetry \cite{anto1}, \cite{anto2}, \cite{dias}, \cite{filipski}, \cite{pivato2}, \cite{golstew4} and synchrony \cite{antonelli2}, \cite{antonelli}, \cite{anto4}, \cite{jeroen}, \cite{dionne1}, \cite{dionne2}, \cite{romano},  \cite{golstew3}, \cite{golstew2}, \cite{torok}, \cite{stewart1}, \cite{pivato}, \cite{wang} in coupled cell networks. Semigroups may well be the natural invariants of coupled cell networks, even more than groups and symmetries. 

Normal forms are computed by applying coordinate transformations \cite{murdock}, \cite{sanders2}, \cite{sanders1}, \cite{sanvermur}. These transformations can be in the phase space of a differential equation, but in our case they take place in the space of functions $f$ and have the form of a series expansion
\[
f\mapsto e^{{\rm ad}^{\Sigma}_g}f = f + {\rm ad}^{\Sigma}_g(f) + \frac{1}{2}({\rm ad}^{\Sigma}_g)^2(f) + \ldots \, .
\]
Here \(f\) is the function to be transformed and normalized, \(g\) generates the coordinate transformation and
\({\rm ad}^{\Sigma}\) denotes a representation, in this case the adjoint representation of the Lie algebra of $f$'s. Although at first sight this may seem a needlessly complicated way to describe coordinate transformations, this ``Lie formalism'' allows for a very flexible theory which streamlines both
the theory and the computations.

The actual computation of the normal form of the function $f$, and in particular the matter of solving homological equations,
will not be entirely standard in the context of networks. Some things remain as in the theory of generic vector fields. For example, we show that the adjoint action of a linear element admits an SN-decomposition that determines a normal form symmetry.
Other aspects may not carry through, such as the applicability of the Jacobson-Morozov lemma to characterize the complement of the image of
the adjoint action of a nilpotent element \cite{sanderscushman}. This is because the Lie algebra of the linear coupled cell network vector fields need not be reductive.  

This paper is organized as follows. After giving a formal definition of a homogeneous coupled cell network in Section \ref{homogeneoussection}, we show in Section \ref{semigroupsection} that semigroups arise naturally in the context of coupled cell networks. In Sections \ref{compositionsection} and \ref{bracketsection} we prove that semigroup network dynamical systems are closed under taking compositions and Lie brackets. Section \ref{normalformsection} explains how to compute the normal form of a network dynamical system, while in Sections \ref{synchronysection} and \ref{symmetrysection} we prove that this normal form inherits both the symmetries and the synchrony spaces of the original network. In Section \ref{SNsection} we investigate the SN-decomposition of a linear coupled cell network vector field. This decomposition determines the normal form symmetry. Section \ref{coveringsection} describes the aforementioned fundamental network. In Section \ref{examplessection} we actually compute the normal forms of some simple but interesting coupled cell networks, thus demonstrating that a coupled cell network structure can force anomalous steady state bifurcations. Finally, we show in Section \ref{generalizationssection} that our theory is also applicable to non-homogeneous or ``colored'' networks that display the structure of a semigroupoid.  

Issues that we do not touch in this paper but aim to treat in subsequent work include: 
\begin{itemize}
\item[{\bf 1.}] The development of a linear algebra of semigroup coupled cell systems in order to define for example a ``semigroup network Jordan normal form''.
\item[{\bf 2.}] Application of the results in this paper to semigroup networks that arise in applications, such as feed-forward motifs.
\item[{\bf 3.}] Understanding the impact of ``input symmetries'' on bifurcations and normal forms.
\end{itemize}
 
 \section{Homogeneous coupled cell networks}\label{homogeneoussection}
 We shall be interested in dynamical systems with a coupled cell network structure. Such a structure can be determined in various ways \cite{field}, \cite{golstew}, \cite{torok}, \cite{pivato}, but we choose to describe it here by means of a collection of distinct maps
$$\Sigma=\{\sigma_1, \ldots, \sigma_n\}\ \mbox{with} \ \sigma_1,\ldots, \sigma_n: \{1, \ldots, N\}\to\{1,\ldots, N\}\, .$$
The collection $\Sigma$ has the interpretation of a network with $1\leq N < \infty$ cells. Indeed, it defines a directed multigraph with $N$ vertices and precisely $n$ arrows pointing into each vertex, where the arrows pointing towards vertex $1\leq i \leq N$ emanate from the vertices $\sigma_1(i), \ldots, \sigma_n(i)$. The number $n$ of incoming arrows per vertex is sometimes called the {\it valence} of the network. 

In a network dynamical system we think of every vertex $1\leq i\leq N$ in the network as a cell, of which the state is determined by a variable $x_i$ that takes values in a vector space $V$.
 
\begin{definition}\label{networkdefinition}
Let $\Sigma=\{\sigma_1, \ldots, \sigma_n\}$ be a collection of $n$ distinct maps on $N$ elements, $V$ a finite dimensional real vector space and $f: V^n\to V$ a smooth function. Then we define
\begin{align}\label{networkvectorfield}
\gamma_f:V^N\to V^N \ \mbox{by}\ (\gamma_f)_i(x):=f(x_{\sigma_1(i)}, \ldots, x_{\sigma_n(i)})\  \mbox{for}\ 1 \leq i \leq N. 
\end{align}
Depending on the context, we will say that $\gamma_f$ is a {\it homogeneous coupled cell network map} or a {\it homogeneous coupled cell network vector field} subject to $\Sigma$.
\end{definition}
In the literature, $\gamma_f$ is also called an {\it admissible} map/vector field for the network $\Sigma$.
 
Dynamical systems with a coupled cell network structure arise when we iterate the map $\gamma_f$ or integrate the vector field that it defines. The iterative dynamics on $V^N$ has the special property that the state of cell $i$ at time $m+1$ depends only on the states of the cells $\sigma_1(i), \ldots, \sigma_n(i)$ at time $m$:
\begin{align}\label{map}
x^{(m+1)}=\gamma_f(x^{(m)})\ \mbox{if and only if} \ x_i^{(m+1)} = f(x_{\sigma_1(i)}^{(m)}, \ldots, x^{(m)}_{\sigma_n(i)}) \ \mbox{for} \ 1\leq i\leq N.
\end{align}
The continuous-time dynamical system on $V^N$ displays the same property infinitesimally: it is determined by the ordinary differential equations  
 \begin{align}\label{diffeqn}
\dot x = \gamma_f(x)\ \mbox{if and only if} \ \dot x_i =  f(x_{\sigma_1(i)}, \ldots, x_{\sigma_n(i)}) \ \mbox{for} \ 1\leq i\leq N.
\end{align}
We aim to understand how the network structure of $\gamma_f$ impacts these dynamical systems.

 \begin{example}\label{ex1}
 An example of a directed multigraph is shown in Figure \ref{pict2}, where the number of cells is $N=3$ and valence is equal to $n=2$. The maps $\sigma_1$ and $\sigma_2$ are given by 
 \begin{align}\nonumber
& \sigma_1(1)=1, \sigma_1(2)=2, \sigma_1(3)=3\, , \\ \nonumber
& \sigma_2(1)=1, \sigma_2(2)=1, \sigma_2(3)=2\, . 
 \end{align}
 \begin{figure}[ht]\renewcommand{\figurename}{\rm \bf \footnotesize Figure} 
\centering
\begin{tikzpicture}[->,>=stealth',shorten >=1pt,auto,node distance=2cm,
                    thick,main node/.style={circle,draw,font=\sffamily\Large\bfseries}]

  \node[main node] (1) {$x_1$};
  \node[main node] (2) [below of=1] {$x_2$};
  \node[main node] (3) [below of=2] {$x_3$};

   \node[main node] (4) [right of=1] {$x_1$};
  \node[main node] (5) [below of=4] {$x_2$};
  \node[main node] (6) [below of=5] {$x_3$};

  \path[every node/.style={font=\sffamily\small}]
    (2) edge [loop above] node {} (2)
 (3) edge [loop above] node {} (3)
(1) edge [loop above] node {} (1)
    
    (5) edge node {} (6)
 (4) edge  node {} (5)
(4) edge [loop above] node {} (4)
    ;
\end{tikzpicture}
\caption{\footnotesize {\rm The collection $\{\sigma_1, \sigma_2\}$ depicted as a directed multigraph.}}
\label{pict2}
\end{figure}
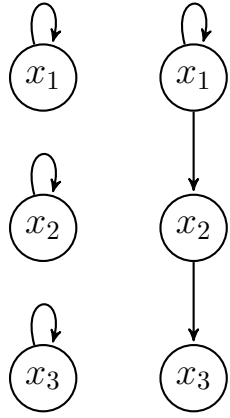

\noindent A coupled cell network map/vector field subject to $\{\sigma_1, \sigma_2\}$ is of the form
\begin{align}\label{examp1}
\gamma_f(x_1, x_2, x_3)=(f(x_1,x_1), f(x_2, x_1), f(x_3, x_2))\, .
\end{align}
This network has obtained some attention \cite{elmhirst}, \cite{curious}, \cite{claire2}, \cite{claire}, \cite{leite} because it supports an anomalous codimension-one nilpotent double Hopf bifurcation when $\dim V=2$.
\end{example}

\begin{example}\label{ex2} In this example we let $\sigma_1, \sigma_2$ be as in Example \ref{ex1} and we also define $\sigma_3$ as
$$\sigma_3(1)=1, \sigma_3(2)=1, \sigma_3(3)=1\, .$$
The network defined by $\{\sigma_1, \sigma_2, \sigma_3\}$ is depicted in Figure \ref{pict3}.
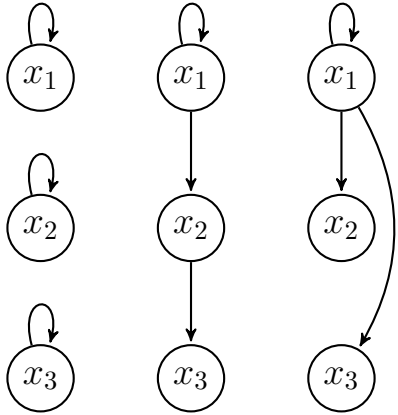
\begin{figure}[ht]\renewcommand{\figurename}{\rm \bf \footnotesize Figure}
\centering
\begin{tikzpicture}[->,>=stealth',shorten >=1pt,auto,node distance=2cm,
                    thick,main node/.style={circle,draw,font=\sffamily\Large\bfseries}]

  \node[main node] (1) {$x_1$};
  \node[main node] (2) [below of=1] {$x_2$};
  \node[main node] (3) [below of=2] {$x_3$};

   \node[main node] (4) [right of=1] {$x_1$};
  \node[main node] (5) [below of=4] {$x_2$};
  \node[main node] (6) [below of=5] {$x_3$};
  \node[main node] (7) [right of =4] {$x_1$};
  \node[main node] (8) [below  of=7] {$x_2$};
  \node[main node] (9) [below of=8] {$x_3$};

  \path[every node/.style={font=\sffamily\small}]
   (2) edge [loop above] node {} (2)
 (3) edge [loop above] node {} (3)
(1) edge [loop above] node {} (1)
    
    (5) edge node {} (6)
 (4) edge  node {} (5)
(4) edge [loop above] node {} (4)
 (7) edge [loop above] node {} (7)
 (7) edge node {} (8)
(7) edge [bend left] node {} (9)
    ;

\end{tikzpicture}
\caption{\footnotesize {\rm The collection $\{\sigma_1, \sigma_2, \sigma_3\}$ depicted as a directed multigraph.}}
\label{pict3}
\end{figure}

\noindent A coupled cell network map/vector field subject to $\{\sigma_1, \sigma_2, \sigma_3\}$ has the form
\begin{align}\label{examp2}
\gamma_g(x_1, x_2, x_3)=(g(x_1,x_1, x_1), g(x_2, x_1, x_1), g(x_3, x_2, x_1))\, .
\end{align}
We remark that this example is a generalization of Example \ref{ex1}: if $\gamma_f$ is as in Example \ref{ex1} and if we define $g(X_1, X_2, X_3):=f(X_1, X_2)$, then $\gamma_g=\gamma_f$. In other words, (\ref{examp1}) arises as a special case of (\ref{examp2}).
\end{example}

\section{Semigroups}\label{semigroupsection}
A first and obvious difficulty that arises in the study of coupled cell network dynamical systems is that the composition $\gamma_f\circ\gamma_g$ of two coupled cell network maps with an identical network structure may not have that same network structure. 

Dynamically, this implies for example that the equation $\gamma_f(x)=x$ for the steady states of $\gamma_f$ and the equation $(\gamma_f)^m(x) =x$ for its periodic solutions may have quite a different nature.  
We illustrate this phenomenon in the following example:
\begin{example}\label{excomposition}
Again, let $N=3$ and let $\sigma_1, \sigma_2, \sigma_3$ be defined as in Examples \ref{ex1} and \ref{ex2}. 
If
\begin{align}\nonumber
& \gamma_f(x_1, x_2, x_3)=(f(x_1, x_1), f(x_2, x_1), f(x_3, x_2))\, , \\
\nonumber
& \gamma_g(x_1, x_2, x_3)=(g(x_1, x_1), g(x_2, x_1), g(x_3, x_2))\, ,
\end{align} 
are coupled cell network maps subject to $\{\sigma_1, \sigma_2\}$, then the composition
$$(\gamma_f\circ \gamma_g)(x_1, x_2, x_3)= (f(g(x_1, x_1), g(x_1, x_1)), f(g(x_2, x_1), g(x_1, x_1)), f(g(x_3, x_2), g(x_2, x_1))) $$
in general is not a coupled cell network map subject to $\{\sigma_1, \sigma_2\}$.

On the other hand,  
when $\gamma_f$ and $\gamma_g$ are network maps subject to $\{\sigma_1, \sigma_2, \sigma_3\}$, i.e. 
\begin{align}\nonumber
&\gamma_f(x_1, x_2, x_3) = (f(x_1, x_1, x_1), f(x_2, x_1, x_1), f(x_3, x_2, x_1))\, , \\ &\gamma_g(x_1, x_2, x_3) = (g(x_1, x_1, x_1), g(x_2, x_1, x_1), g(x_3, x_2, x_1))\, ,
\nonumber
\end{align}
then it holds that 
\begin{align} \nonumber
&(\gamma_f\circ \gamma_g)_1(x_1, x_2, x_3)= f(g(x_1, x_1, x_1), g(x_1, x_1, x_1), g(x_1, x_1, x_1) \, , \\ \nonumber
&(\gamma_f \circ \gamma_g)_2(x_1, x_2, x_3)=f(g(x_2, x_1,x_1), g(x_1, x_1, x_1), g(x_1, x_1, x_1))\, ,\\ \nonumber
& (\gamma_f \circ \gamma_g)_3 (x_1, x_2, x_3)= f(g(x_3, x_2, x_1), g(x_2, x_1, x_1), g(x_1, x_1, x_1)))\, .
\end{align}
This demonstrates that $\gamma_f\circ \gamma_g$ is also a coupled cell network map subject to $\{\sigma_1, \sigma_2, \sigma_3\}$. Indeed, 
$\gamma_f\circ \gamma_g=\gamma_h$, where
$$h(X_1,X_2, X_3) =  f(g(X_1, X_2, X_3), g(X_2, X_3, X_3), g(X_3, X_3, X_3))\, .$$
\end{example}
To understand when, in general, the composition of two coupled cell network maps is again a coupled cell network map, we compute that
\begin{align}\label{FafterG}
(\gamma_f\circ \gamma_g)_i(x)= f(\ldots,  (\gamma_g)_{\sigma_j(i)}(x), \ldots) = f(\ldots, g(x_{\sigma_1(\sigma_j(i))}, \ldots, x_{\sigma_n(\sigma_j(i))}), \ldots) \, .
\end{align}
The right hand side of (\ref{FafterG}) is an $i$-independent function of $(x_{\sigma_1(i)}, \ldots, x_{\sigma_n(i)})$ precisely when for all $1\leq j_1, j_2\leq n$ and all $1\leq i\leq N$ it holds that $\sigma_{j_1}(\sigma_{j_2}(i))=\sigma_{j_3}(i)$ for some $1\leq j_3\leq n$. In other words, 
 $\gamma_f\circ\gamma_g$ is a coupled cell network map when $\Sigma$ is a semigroup:
\begin{definition}
 We say that $\Sigma=\{\sigma_1, \ldots, \sigma_n\}$ is a {\it semigroup} if for all $1\leq j_1, j_2\leq n$ there is a unique $1\leq j_3\leq n$ such that 
$\sigma_{j_1}\circ \sigma_{j_2}= \sigma_{j_3}$.
 \end{definition}
Viewing $\Sigma$ as a directed multigraph, the condition that it is a semigroup just means that this directed multigraph is closed under the backward concatenation of arrows.  
 
Of course, an arbitrary collection $\Sigma=\{\sigma_1,\ldots,\sigma_n\}$ need not be a semigroup. Even so, $\Sigma$ generates a unique smallest semigroup $$\Sigma'=\{\sigma_1, \ldots, \sigma_n, \sigma_{n+1},\ldots, \sigma_{n'}\} \ \mbox{that contains}\ \Sigma\, .$$
It is clear that every coupled cell network map $\gamma_f$ subject to $\Sigma$ is also a coupled cell network map subject to the semigroup $\Sigma'$. Indeed, if we define 
$$f'(X_1, \ldots, X_n, X_{n+1}, \ldots, X_{n'}) := f(X_1, \ldots, X_n)$$ 
then it obviously holds that 
$$(\gamma_{f'})_i(x) = f'(x_{\sigma_1(i)}, \ldots, x_{\sigma_n(i)}, x_{\sigma_{n+1}(i)}, \ldots, x_{\sigma_{n'}(i)}) = f(x_{\sigma_1(i)}, \ldots, x_{\sigma_n(i)}) = (\gamma_{f})_i(x)\ .$$
We thus propose to augment $\Sigma$ to the semigroup $\Sigma'$ and to think of every coupled cell network map subject to $\Sigma$ as a (special case of a) coupled cell network map subject to $\Sigma'$.  

\begin{example}\label{example1}
Again, let $N=3$ and let $\sigma_1, \sigma_2, \sigma_3$ be defined as in Examples \ref{ex1} and \ref{ex2}. It holds that $\sigma_2^2=\sigma_3$, so the collection $\{\sigma_1, \sigma_2\}$ is not a semigroup. On the other hand, one quickly computes that the composition table of $\{\sigma_1, \sigma_2, \sigma_3\}$ is given by
$$\begin{array}{c|ccc} 
\circ & \sigma_1 & \sigma_2 &\sigma_3 \\
\hline
\sigma_1 & \sigma_1 &\sigma_2 & \sigma_3 \\
\sigma_2 & \sigma_2 &\sigma_3 & \sigma_3  \\
\sigma_3 & \sigma_3 & \sigma_3 & \sigma_3
\end{array} \begin{array}{l} \\ \\ \\ .\end{array}
$$
This shows that $\{\sigma_1, \sigma_2, \sigma_3\}$ is closed under composition and hence is the smallest semigroup containing $\{\sigma_1, \sigma_2\}$. 
\end{example}

\section{Composition of network maps}\label{compositionsection}
To understand better how network maps behave under composition and in order to simplify our notation, let us define the maps
$$\pi_{i}: V^N\to V^n\ \mbox{by}\ \pi_i(x_1, \ldots, x_N):=(x_{\sigma_1(i)}, \ldots, x_{\sigma_n(i)}) \ \mbox{for} \ 1\leq i\leq N\, .$$
This definition allows us to write (\ref{networkvectorfield}) simply as
\begin{align}\label{networkdef}
(\gamma_f)_i:=f\circ \pi_i\ .
\end{align}
Expression (\ref{FafterG}) moreover turns into the formula
\begin{align}\label{compositionformula}
(\gamma_f\circ \gamma_g)_i = f\circ (g\circ \pi_{\sigma_1(i)}\times \ldots \times g\circ \pi_{\sigma_n(i)})\ . 
\end{align}
The following technical result helps us write the right hand side of (\ref{compositionformula}) in the form $h\circ\pi_i$ for some function $h:V^n\to V$, whenever $\Sigma$ is a semigroup.
\begin{theorem}\label{Amaps}
Let $\Sigma=\{\sigma_1,\ldots, \sigma_n\}$ be a semigroup. 
Then for all $1\leq j \leq n$ there exists a linear map 
$A_{\sigma_j} : V^n\to V^n$ with the property that
\begin{align}\label{formulaA} A_{\sigma_j}\circ \pi_{i} = \pi_{\sigma_j(i)} \ \mbox{for all} \ 1\leq i \leq N\ \mbox{and all}\ 1\leq j \leq n\, .
\end{align}
Moreover, it holds that $A_{\sigma_{j_1}}\circ A_{\sigma_{j_2}}= A_{ \sigma_{j_1}\circ \sigma_{j_2}}$ for all $1\leq j_1, j_2 \leq n$.
\end{theorem}
\begin{proof}
Because $\Sigma$ is a semigroup, we can associate to each map $\sigma_j\in \Sigma$ a unique map
\begin{align}\label{tildesigma}
\widetilde \sigma_j: \{1,\ldots, n\}\to\{1, \ldots, n\}\ \mbox{defined via the formula}\ \sigma_{\widetilde \sigma_j(k)} = \sigma_j \circ \sigma_k\ .
\end{align}
We now define the map $A_{\sigma_j}: V^n\to V^n$ as
\begin{align}\label{Amapdef}
 A_{\sigma_j}(X_1, \ldots, X_{n}) :=(X_{\widetilde \sigma_1(j)}, \ldots, X_{\widetilde \sigma_n(j)})\, .
\end{align}
With this definition it holds that
\begin{align}
(A_{\sigma_j}\circ \pi_i)(x)& =A_{\sigma_j}(x_{\sigma_1(i)}, \ldots, x_{\sigma_n(i)}) = (x_{\sigma_{\widetilde\sigma_1(j)}(i)}, \ldots, x_{\sigma_{\widetilde\sigma_n(j)}(i)}) \nonumber \\ \nonumber  
& = (x_{\sigma_1(\sigma_j(i))}, \ldots, x_{\sigma_n(\sigma_j(i))}) = \pi_{\sigma_j(i)}(x)\, .
\end{align}
Remarking moreover that
$$\sigma_{\widetilde \sigma_{\widetilde \sigma_{k}(j_1)}(j_2)} = \sigma_{\widetilde \sigma_{k}(j_1)}\circ \sigma_{j_2}= \sigma_k\circ\sigma_{j_1}\circ\sigma_{j_2} =  \sigma_k\circ\sigma_{\widetilde \sigma_{j_1}(j_2)}= \sigma_{\widetilde \sigma_k(\widetilde \sigma_{j_1}(j_2))} \, ,$$
and hence that $\widetilde \sigma_{\widetilde \sigma_{k}(j_1)}(j_2) = \widetilde \sigma_k(\widetilde \sigma_{j_1}(j_2))$ for all $1\leq k\leq n$, we also find that
\begin{align}\nonumber
(&A_{\sigma_{j_1}}\circ A_{\sigma_{j_2}}) (X_1, \ldots, X_n) = A_{\sigma_{j_1}}(X_{\widetilde \sigma_{1}(j_2)}, \ldots, X_{\widetilde \sigma_{n}(j_2)})=  (X_{\widetilde \sigma_{\widetilde \sigma_{1}(j_1)}(j_2)}, \ldots, X_{\widetilde \sigma_{\widetilde \sigma_{n}(j_1)}(j_2)})   \\ \nonumber 
&= (X_{\widetilde \sigma_{1}(\widetilde \sigma_{j_1}(j_2))}, \ldots, X_{\widetilde \sigma_{n}(\widetilde \sigma_{j_1}(j_2))}) = A_{\sigma_{\widetilde \sigma_{j_1}(j_2)}}(X_1, \ldots, X_n)= A_{\sigma_{j_1}\circ \sigma_{j_2}}(X_1, \ldots, X_n) \ .
\end{align}
This proves the theorem.
\end{proof}
The identity $A_{\sigma_{j_1}}\circ A_{\sigma_{j_2}}=A_{\sigma_{j_1}\circ\sigma_{j_2}}$ expresses that the $A_{\sigma_j}$ form a representation of the semigroup $\Sigma$. Using this representation we obtain:
\begin{theorem}\label{composition} Let $\Sigma=\{\sigma_1, \ldots, \sigma_n\}$ be a semigroup. Define for $f, g:V^n\to V$ the function
\begin{align}\label{alsocomposition}
f\circ_{\Sigma}g:V^n\to V \ \mbox{by}\  
f\circ_{\Sigma} g:=f\circ ((g\circ A_{\sigma_1})\times\ldots \times (g\circ A_{\sigma_n}))\, .
\end{align}
Then
$$\gamma_f\circ \gamma_g = \gamma_{f\circ_{\Sigma} g}\, .$$
\end{theorem}
\begin{proof}
From formula (\ref{compositionformula}) and Theorem \ref{Amaps}.
\end{proof}
Theorem \ref{composition} reveals once more that if $\Sigma$ is a semigroup, then the composition of two coupled cell network maps $\gamma_f$ and $\gamma_g$ is again a coupled cell network map, namely $\gamma_{f\circ_{\Sigma}g}$. More importantly, it shows how to compute $f\circ_{\Sigma}g$ ``symbolically'', i.e. using only the functions $f$ and $g$ and a representation of the network semigroup.

The final result of this section ensures that the ``symbolic composition'' $\circ_{\Sigma}$ makes the space $C^{\infty}(V^n,V)$ into an associative algebra.
\begin{lemma}\label{associative}
$$(f\circ_{\Sigma} g) \circ_{\Sigma} h = f\circ_{\Sigma} (g \circ_{\Sigma} h) \, .$$
\end{lemma}
\begin{proof}
\begin{align}\nonumber
(f\circ_{\Sigma} g) \circ_{\Sigma} h (X) & = (f\circ_{\Sigma}g)(\ldots, h(A_{\sigma_k}X), \ldots) = f(\ldots, g(A_{\sigma_j}(\ldots, h(A_{\sigma_k}X),\ldots)), \ldots) =\\ \nonumber 
f(\ldots&,g(\ldots, h(A_{\sigma_{\widetilde \sigma_k(j)}}X),\ldots)=
 f(\ldots, g(\ldots, h(A_{\sigma_k\circ \sigma_j}X), \ldots) = \\ \nonumber  f(\ldots, g(\ldots, h(&A_{\sigma_k}A_{\sigma_j}X), \ldots),\ldots)
= 
f(\ldots, (g\circ_{\Sigma} h)(A_{\sigma_j}X), \ldots) =f\circ_{\Sigma} (g \circ_{\Sigma} h ) (X)\, .
\end{align}  
\end{proof}
With Lemma \ref{associative} at hand, Theorem \ref{composition} just means that the linear map
$$\gamma:C^{\infty}(V^n,V) \to C^{\infty}(V^N, V^N)\ \mbox{that sends}\ f \ \mbox{to}\ \gamma_f$$
is a homomorphism of associative algebras.

\begin{example} \label{ex3} Again, let $N=3$ and let $\sigma_1, \sigma_2, \sigma_3$ be defined as in Examples \ref{ex1} and \ref{ex2}. We recall that the composition table of $\{\sigma_1, \sigma_2, \sigma_3\}$ was given in Example \ref{example1}. The rows of this table express that 
\begin{align}\nonumber
\widetilde \sigma_1(1)=1, \widetilde \sigma_1(2)=2, \widetilde \sigma_1(3)=3\, ,\\ \nonumber 
\widetilde \sigma_2(1)=2, \widetilde \sigma_2(2)=3, \widetilde \sigma_2(3)=3\, ,\\ \nonumber
\widetilde \sigma_3(1)=3, \widetilde \sigma_3(2)=3, \widetilde \sigma_3(3)=3\, .
\end{align}
 This implies in particular that
\begin{align}\nonumber
& A_{\sigma_1}(X_1, X_2, X_3) = (X_1, X_2, X_3) \, , \\ \nonumber & A_{\sigma_2}(X_1, X_2, X_3) = (X_2, X_3, X_3) \, , \\ 
& A_{\sigma_3}(X_1, X_2, X_3) = (X_3, X_3, X_3) \, . \nonumber
\end{align}
Substitution in (\ref{alsocomposition}) therefore yields that
$$f\circ_{\Sigma}g\, (X_1,X_2, X_3) = f(g(X_1, X_2, X_3), g(X_2, X_3, X_3), g(X_3, X_3, X_3))\, .$$
We conclude that $f\circ_{\Sigma} g$ equals the function $h$ found in Example \ref{excomposition}.
\end{example}

\begin{remark}\label{tildesigmaremark}
The defining relation
\begin{align} \nonumber
\sigma_{\widetilde \sigma_j(k)} = \sigma_j \circ \sigma_k\ \mbox{for} \ \widetilde \sigma_j:\{1, \ldots, n\}\to \{1, \ldots, n\}\,
\end{align}
expresses that the map $\widetilde \sigma_j$ describes the left-multiplicative behavior of $\sigma_j$. The computation
$$\sigma_{\widetilde{\sigma_{j_1}\circ\sigma_{j_2}}(k)}=\sigma_{j_1}\circ\sigma_{j_2} \circ\sigma_{k}= \sigma_{j_1}\circ\sigma_{\widetilde \sigma_{j_2}(k)}=  \sigma_{\widetilde \sigma_{j_1}(\widetilde \sigma_{j_2}(k))}\, $$
moreover reveals that
$$\widetilde{\ \sigma_{j_1}\circ\sigma_{j_2}\ } = \widetilde \sigma_{j_1}\circ \widetilde \sigma_{j_2}\ \mbox{for all}\ 1\leq j_1, j_2\leq n\, .$$
This means in particular that the collection $\{\widetilde \sigma_1, \ldots, \widetilde \sigma_n\}$ is closed under composition. 
 
The maps $\widetilde \sigma_1, \ldots, \widetilde \sigma_n$ will play an interesting role in this paper. In fact, we will show in Section \ref{coveringsection} that they are themselves the network maps of a certain ``fundamental network'' that fully determines the fate of all network dynamical systems subject to $\Sigma$.
\end{remark}

\begin{remark}
For a map $\sigma:\{1,\ldots, n\}\to\{1, \ldots, m\}$ let us denote by $\lambda_{\sigma}:V^{m}\to V^{n}$ the linear map
$$\lambda_{\sigma}(X_1, \ldots, X_{m}) := (X_{\sigma(1)}, \ldots , X_{\sigma(n)})\ .$$
This means that the matrix of the map $\lambda_{\sigma}$ has precisely one $\mbox{id}_V$ on each row and zeroes elsewhere. 
We will denote the space of such maps by 
$$\Lambda(m, n):=\left\{\lambda_{\sigma}: V^{m}\to V^{n}\, | \, \sigma:\{1,\ldots, n\}\to\{1, \ldots, m\} \right\}\, .$$
One quickly checks that the assignment $\lambda:\sigma\mapsto\lambda_{\sigma}$ is contravariant. More precisely, if $\sigma:\{1,\ldots, n\}\to\{1,\ldots, m\}$ and $\tau:\{1,\ldots, m\}\to\{1,\ldots, l\}$, then 
$$\lambda_{\sigma}\circ \lambda_{\tau}  =  \lambda_{\tau\circ\sigma}\, .$$
In particular, $\Lambda(m, m)$ is a semigroup and the map $\sigma\mapsto \lambda_{\sigma}$ an anti-homomorphism from the semigroup of all maps from $\{1,\ldots, m\}$ to itself to the semigroup $\Lambda(m, m)$.

The maps $\pi_i$ and $A_{\sigma_j}$ defined above are examples of such maps: 
\begin{align}
\pi_i=&\lambda_{\sigma^i}\ \mbox{with} \ \sigma^i:\{1,\ldots, n\}\to \{1,\ldots, N\}\ \mbox{defined as}\ \sigma^i(j):=\sigma_j(i)\, .\nonumber \\ \nonumber
A_{\sigma_j}=&\lambda_{\widetilde \sigma^j}\ \mbox{with}\ \widetilde \sigma^j:\{1,\ldots, n\}\to \{1,\ldots, n\}\ \mbox{defined by}\ \widetilde \sigma^j(k):=\widetilde \sigma_k(j)\, .
\end{align}
This observation can be used give a proof of Theorem \ref{Amaps} that is free of coordinates:

\begin{proofof}\ [{\bf Of Theorem \ref{Amaps} without coordinates}]:
We observe that
$$\sigma^i(\widetilde \sigma^j(k))=\sigma^i(\widetilde \sigma_k(j)) = \sigma_{\widetilde \sigma_k(j)}(i)= (\sigma_k\circ\sigma_j)(i)= \sigma_k(\sigma_j(i)) = \sigma^{\sigma_j(i)}(k)$$
and hence that $\sigma^i\circ\widetilde \sigma^j= \sigma^{\sigma_j(i)}$. Using this, we find that
$$A_{\sigma_j}\circ\pi_i=\lambda_{\widetilde \sigma^j}\circ\lambda_{\sigma^i}=\lambda_{\sigma^i\circ\widetilde \sigma^j} = \lambda_{\sigma^{\sigma_j(i)}} =\pi_{\sigma_j(i)}\, .$$
Similarly, the computation
\begin{align}\nonumber
&\sigma_{(\widetilde \sigma^{j_2}\circ \widetilde \sigma^{j_1})(k)}= \sigma_{\widetilde \sigma^{j_2}(\widetilde \sigma^{j_1}(k))} =  \sigma_{\widetilde \sigma^{j_2}(\widetilde \sigma_k(j_1))} = \sigma_{\widetilde \sigma_{\widetilde \sigma_k(j_1)}(j_2)} \\ \nonumber
=\sigma_{\widetilde \sigma_k(j_1)}\circ&\sigma_{j_2} = \sigma_k\circ\sigma_{j_1}\circ\sigma_{j_2} = \sigma_k\circ \sigma_{\widetilde \sigma_{j_1}(j_2)}  =\sigma_{\widetilde \sigma_k(\widetilde \sigma_{j_1}(j_2))}  = \sigma_{\widetilde \sigma^{\widetilde \sigma_{j_1}(j_2)}(k)}
\end{align}
reveals that
$$\widetilde \sigma^{j_2} \circ \widetilde \sigma^{j_1} = \widetilde \sigma^{\widetilde \sigma_{j_1}(j_2)}\, .$$
As a consequence,
$$A_{\sigma_{j_1}}\circ A_{\sigma_{j_2}} = \lambda_{\widetilde \sigma^{j_1}}\circ  \lambda_{\widetilde \sigma^{j_2}} =  \lambda_{\widetilde \sigma^{j_2} \circ \widetilde \sigma^{j_1}}=  \lambda_{\widetilde \sigma^{\widetilde \sigma_{j_1}(j_2)}} = A_{\sigma_{\widetilde \sigma_{j_1}(j_2)}}=A_{\sigma_{j_1}\circ \sigma_{j_2}}\, .$$
\end{proofof}
\noindent Unfortunately, this coordinate free proof of Theorem \ref{Amaps} is relatively long.
\end{remark}
\section{A coupled cell network bracket}\label{bracketsection}
We will now think of $\gamma_f:V^N\to V^N$ as a vector field that generates the differential equation $$\dot x = \gamma_f(x)\ .$$
We suggestively denote by $e^{t\gamma_f}$ the time-$t$ flow of the vector field $\gamma_{f}$ on $V^N$ and by $(e^{t\gamma_g})_*\gamma_f $ the pushforward of the vector field $\gamma_f$ under the time-$t$ flow of $\gamma_g$.
We recall that the Lie bracket of $\gamma_f$ and $\gamma_g$ is then the vector field $[\gamma_f, \gamma_g]:V^N\to V^N$ defined as 
\begin{align}\label{defnormalbracket}
[\gamma_f, \gamma_g](x):= \left.\frac{d}{dt}\right|_{t=0}\!\!\!  (e^{t\gamma_f})_*\gamma_g = D\gamma_f(x) \cdot \gamma_g(x) - D\gamma_g(x)\cdot \gamma_f(x)\ .
 \end{align}
  The main result of this section is that if $\Sigma$ is a semigroup, then the collection of coupled cell network vector fields is closed under taking Lie brackets. 

\begin{theorem}\label{brackettheorem}
Let $\Sigma=\{\sigma_1, \ldots, \sigma_n\}$ be a semigroup and let the $A_{\sigma_j}:V^n\to V^n$ be as in Theorem \ref{Amaps}. Define, for $f, g:V^n\to V$, the function $[f,g]_{\Sigma}:V^n\to V$ by
\begin{align}\label{bracketdef}
[f,g]_{\Sigma}:= \sum_{j=1}^n D_jf \cdot (g \circ A_{\sigma_j}) -  D_jg \cdot (f\circ A_{\sigma_j}) \ .
\end{align}
Then
\begin{align}\label{bracketprop}
[\gamma_f, \gamma_g]=\gamma_{[f,g]_{\Sigma}}\ .
\end{align}
\end{theorem}
\begin{proof}
We start by remarking that 
\begin{align}
\gamma_f&(x+t\gamma_g(x))_i =  f(\ldots, x_{\sigma_j(i)} + t\gamma_g(x)_{\sigma_j(i)}, \ldots) =\nonumber \\\nonumber &f(\ldots, x_{\sigma_j(i)} + tg(x_{\sigma_1(\sigma_j(i))}, \ldots, x_{\sigma_n(\sigma_j(i))}), \ldots)
=  \\ \nonumber
&f(\ldots ,  x_{\sigma_j(i)} + tg(A_{\sigma_j}(x_{\sigma_1(i)}, \ldots, x_{\sigma_n(i)})), \ldots) \, .
\end{align}
Differentiating this identity with respect to $t$ and evaluating the result at $t=0$ gives that
\begin{align}\nonumber
 (D\gamma_f(x)\cdot\gamma_g(x))_i =  \sum_{j=1}^n D_jf(x_{\sigma_1(i)}, \ldots, x_{\sigma_n(i)})  g(A_{\sigma_j}(x_{\sigma_1(i)}, \ldots, x_{\sigma_n(i)})) \ . \nonumber
\end{align}
This proves that $(D\gamma_{f}\cdot \gamma_g)_i= \left( \sum_{j=1}^n D_jf \cdot (g \circ A_{\sigma_j})\right)\circ \pi_i$, and hence that $D\gamma_f\cdot \gamma_g$ is a coupled cell network vector field.
With a similar computation for $D\gamma_g\cdot \gamma_f$, we thus find that the Lie bracket between $\gamma_f$ and $\gamma_g$ is given by
$$[\gamma_f, \gamma_g]_i = (D\gamma_{f}\cdot \gamma_g)_i - (D\gamma_{g}\cdot \gamma_f)_i = \sum_{j=1}^n \left[ D_jf\cdot (g\circ A_{\sigma_j}) -  D_jg\cdot (f \circ A_{\sigma_j})\right]\circ \pi_i = [f,g]_{\Sigma}\circ \pi_i\, .$$
This proves the theorem.
\end{proof}
Lemma \ref{bracketlemma} below states that the ``symbolic bracket''  $[\cdot, \cdot]_{\Sigma}$ is a Lie bracket.
\begin{lemma} \label{bracketlemma}
The bracket $[\cdot, \cdot]_{\Sigma}$ makes $C^{\infty}(V^n, V)$ a Lie algebra. Moreover, the linear map 
$$\gamma:C^{\infty}(V^n,V) \to C^{\infty}(V^N, V^N) \ \mbox{that sends}\ f \ \mbox{to} \ \gamma_f\ \mbox{is a Lie algebra homomorphism}.$$
\end{lemma}
 \begin{proof} 
Anti-symmetry of $[\cdot, \cdot]_{\Sigma}$ is clear from formula (\ref{bracketdef}). The Jacobi identity 
$$ [f, [g,h]_{\Sigma}]_{\Sigma} + [g, [h,f]_{\Sigma}]_{\Sigma} + [h, [f,g]_{\Sigma}]_{\Sigma} = 0$$
follows from a somewhat lengthy computation as follows. First of all, because 
$$A_{\sigma_j}(\ldots, h(A_{\sigma_k}X), \ldots) = (\ldots, h(A_{\sigma_{\widetilde \sigma_k(j)}}X),\ldots) = (\ldots, h(A_{\sigma_k} A_{\sigma_j}X), \ldots)\, ,$$ 
we find that
\begin{align}\nonumber
(g\circ A_{\sigma_j})(X + t(&\ldots, h(A_{\sigma_k}X),\ldots)) = g(A_{\sigma_j}X + tA_{\sigma_j}(\ldots, h(A_{\sigma_k}X),\ldots)) \\ \nonumber  
 & = g(A_{\sigma_j}X +  t(\ldots, h(A_{\sigma_k}A_{\sigma_j}X), \ldots)) \, .
 \end{align}
Differentiating this identity with respect to $t$ and evaluating the result at $t=0$ yields that
$$ \sum_{k=1}^n D_k(g\circ A_{\sigma_j}) \cdot (h\circ A_{\sigma_k})= \sum_{k=1}^n (D_kg\circ A_{\sigma_j} ) \cdot (h\circ A_{\sigma_k}\circ A_{\sigma_j})\, .$$
With this in mind, we now compute that
\begin{align}\nonumber  [f, [g,h]_{\Sigma}]_{\Sigma} = \left.\sum\right._{j=1}^n & D_jf\cdot ([g,h]_{\Sigma}\circ A_{\sigma_j}) - D_j[g,h]_{\Sigma} \cdot (f\circ A_{\sigma_j})  
\\
\nonumber  = \left. \sum \right._{j,k=1}^n\  D_jf & \cdot (D_kg\circ A_{\sigma_j})\cdot(h\circ A_{\sigma_k}\circ A_{\sigma_j}) - D_jf\cdot (D_kh\circ A_{\sigma_j})\cdot(g\circ A_{\sigma_k}\circ A_{\sigma_j})  \\
\nonumber
 - D_kg & \cdot D_j(h\circ A_{\sigma_k})\cdot(f\circ A_{\sigma_j}) + D_kh\cdot D_j(g\circ A_{\sigma_k})\cdot(f\circ A_{\sigma_j})  \\ \nonumber
 - D^2_{j,k} g & \cdot (  h\circ A_{\sigma_k}, f\circ A_{\sigma_j}) + D^2_{j,k} h \cdot ( g\circ A_{\sigma_k}, f\circ A_{\sigma_j})  \\ \nonumber
  =\left. \sum \right._{j, k=1}^n \ D_kf & \cdot D_j(g\circ A_{\sigma_k})\cdot(h\circ A_{\sigma_j}) - D_kf\cdot D_j(h\circ A_{\sigma_k})\cdot(g\circ A_{\sigma_j})   \\
\nonumber
  - D_kg & \cdot D_j(h\circ A_{\sigma_k})\cdot(f\circ A_{\sigma_j}) + D_kh\cdot D_j(g\circ A_{\sigma_k})\cdot(f\circ A_{\sigma_j})  \\ \nonumber
 - D^2_{j,k} g & \cdot ( h\circ A_{\sigma_k}, f\circ A_{\sigma_j}) + D^2_{j,k} h \cdot ( g\circ A_{\sigma_k}, f\circ A_{\sigma_j}) \, .
\end{align} 
Using the symmetry of the second derivatives, the Jacobi identity follows from cyclically permuting $f, g$ and $h$ in the above expression and summing the results. This proves that $C^{\infty}(V^n,V)$ is a Lie algebra. 
Theorem \ref{brackettheorem} means that $\gamma$ is a Lie algebra homomorphism.
 \end{proof}
\begin{example}\label{exzoveel}
Again, let $N=3$ and let $\sigma_1, \sigma_2, \sigma_3$ be defined as in Examples \ref{ex1} and \ref{ex2}. We recall that $A_{\sigma_1}, A_{\sigma_2}$ and $A_{\sigma_3}$ were computed in Example \ref{ex3}. It follows that
\begin{align}\nonumber
[f,g]_{\Sigma}(X)&=D_1f(X_1, X_2, X_3)\cdot g(X_1, X_2, X_3)+D_2f(X_1, X_2, X_3)\cdot g(X_2, X_3, X_3)  \\ \nonumber & + D_3f(X_1, X_2, X_3)\cdot g(X_3, X_3, X_3) - D_1g(X_1, X_2, X_3)\cdot f(X_1, X_2, X_3) \\ \nonumber
& -  D_2g(X_1, X_2, X_3)\cdot f(X_2, X_3, X_3) - D_3g(X_1, X_2, X_3)\cdot f(X_3, X_3, X_3)  \, .
\end{align}
\end{example}

\section{Coupled cell network normal forms}\label{normalformsection}
Normal forms are an essential tool in the study of the dynamics and bifurcations of maps and vector fields near equilibria, cf. \cite{murdock}, \cite{sanvermur}.
In this section we will show that it can be arranged that the normal form of a semigroup coupled cell network is a coupled cell network as well. This normal form can moreover be computed ``symbolically'', i.e. at the level of the function $f$. With Theorem \ref{brackettheorem} at hand, this result is perhaps to be expected. We nevertheless state two illustrative theorems in this section, and sketch their proofs.

We start by making a few standard definitions. First of all, we define for $f\in C^{\infty}(V^n, V)$ the operator ${\rm ad}^{\Sigma}_f:C^{\infty}(V^n, V)\to C^{\infty}(V^n,V)$
by
$${\rm ad}^{\Sigma}_f(g):=[f,g]_{\Sigma}\ .$$
Next, we define for every $k=0,1,2, \ldots$ the finite dimensional subspace  
 \begin{align}
 P^k&:=\{f:V^n \to V \ \mbox{homogeneous polynomial of degree}\ k+1\, \} \subset C^{\infty}(V^n,V) \, .  
 \end{align}
\noindent 
One can observe that $P^0=L(V^N,V)$ and that if $f\in P^k$ and $g\in P^l$, then $[f,g]_{\Sigma}\in P^{k+l}$, as is obvious from formula (\ref{bracketdef}). In particular, we have that
$$\mbox{if} \ f_0\in L(V^n,V)\ \mbox{then}\ {\rm ad}_{f_0}^{\Sigma}: P^k\to P^{k}\, .$$
With this in mind, we formulate the first main result of this section. It essentially states that one may restrict the study of semigroup coupled cell networks near local equilibria to semigroup coupled cell networks of a very specific ``normal form''.

\begin{theorem}[Coupled cell network normal form theorem] \label{normalformtheorem}
Let $\Sigma=\{\sigma_1, \ldots, \sigma_n\}$ be a semigroup, $f\in C^{\infty}(V^n, V)$ and assume that $f(0)=0$. We Taylor expand $f$ as 
$$f= f_0+f_1+ f_2+ \ldots  \ \mbox{with}\ f_k \in P^k\, .$$ 
Let $1\leq r<\infty$ and for every $1\leq k \leq r$, let $N^k\subset P^k$ be a subspace such that 
$$N^k \oplus \left. {\rm im \ ad}^{\Sigma}_{f_0} \right|_{P^k} = P^k\, .$$
Then there exists an analytic diffeomorphism $\Phi$, sending an open neighborhood of $0$ in $V^N$ to an open neighborhood of $0$ in $V^N$, that conjugates the coupled cell network vector field $\gamma_{f}$ to a coupled cell network vector field $\gamma_{\overline f}$ with
$$ \overline f = f_0+ \overline f_1+\overline f_2 + \ldots \ \mbox{and}\ \overline f_k \in N^k  \ \mbox{for all} \ 1\leq k\leq r\, .$$
\end{theorem}
\begin{proof}{\bf [Sketch] } We only sketch a proof without estimates here, because the construction of a normal form by means of ``Lie transformations'' is very well-known.

For $g\in C^{\infty}(V^n,V)$ with $g(0)=0$, the time-$t$ flow $e^{t\gamma_g}$ defines a diffeomorphism of some open neighborhood of $0$ in $V^N$ to another open neighborhood of $0$ in $V^N$. Thus we can consider, for any $f \in C^{\infty}(V^n,V)$, the curve $t \mapsto (e^{t\gamma_{g}})_*\gamma_f \in C^{\infty}(V^N,V^N)$ of pushforward vector fields. This curve satisfies the linear differential equation 
\begin{align}\label{nfdiff}
\frac{d}{dt} (e^{t\gamma_g})_*\gamma_f =  \left. \frac{d}{dh}\right|_{h=0} \!\!\!\!\! (e^{h\gamma_g})_* ((e^{t\gamma_g})_*\gamma_f) = [\gamma_g, (e^{t\gamma_g})_*\gamma_f ] = {\rm ad}_{\gamma_g}((e^{t\gamma_g})_*\gamma_f )\ ,
\end{align}
where the second equality holds by definition of the Lie bracket of vector fields (\ref{defnormalbracket}) and we have used the conventional definition of ${\rm ad}_{\gamma_g}:C^{\infty}(V^N, V^N)\to C^{\infty}(V^N, V^N)$ as
$${\rm ad}_{\gamma_g}(\gamma_f) := [\gamma_g, \gamma_f]\ . $$
Solving the linear differential equation (\ref{nfdiff}) together with the initial condition $(e^{0\gamma_g})_*\gamma_f=\gamma_f$, we find that  the time-$1$ flow of $\gamma_g$ transforms $\gamma_f$ into
$$(e^{\gamma_g})_*\gamma_f = e^{{\rm ad}_{\gamma_g}}(\gamma_f) = \gamma_f + [\gamma_g, \gamma_f]+\frac{1}{2}[\gamma_g, [\gamma_g, \gamma_f]] + \ldots \ .
$$
The main point of this proof is that by Theorem \ref{brackettheorem} the latter expression is also equal to
$$\gamma_{f + [g,f]_{\Sigma}+ \frac{1}{2}[g, [g, f]_{\Sigma}]_{\Sigma} + \ldots} \!=\! \gamma_{e^{{\rm ad}^{\Sigma}_g}(f)}.$$
The diffeomorphism $\Phi$ in the statement of the theorem is now constructed as the composition of a sequence of time-$1$ flows $e^{\gamma_{g_k}}$ $(1 \leq k \leq r)$ of coupled cell network vector fields $\gamma_{g_{k}}$ with $g_k\in  P^k$. We first take $g_1 \in P^1$, so that $\gamma_f$ is transformed by $e^{\gamma_{g_1}}$ into 
$$(e^{\gamma_{g_1}})_*\gamma_f=\gamma_{e^{{\rm ad}_{g_1}}(f)} =\gamma_{f_0 + f_1^1+ f_2^1 + \ldots}$$
in which 
\begin{align}\nonumber
\begin{array}{ll} f^1_1 = f_1+[g_1, f_0]_{\Sigma} & \in P^1 \\
f_2^1 =f_2 + [g_1, f_1]_{\Sigma} + \frac{1}{2}[g_1, [g_1, f_0]_{\Sigma}]_{\Sigma} & \in P^2\\
f_3^1=f_3+\ldots & \in P^3\\
\mbox{etc.}& \ 
\end{array}
\end{align}
It is the fact that $N^1 \oplus \left. {\rm im \ ad}^{\Sigma}_{f_0}\right|_{P^1} =P^1$ that allows us to choose a (not necessarily unique) $g_1\in P^1$ in such a way that 
$$f_1^1=  f_1 + [g_1, f_0]_{\Sigma} = f_1 -{\rm ad}^{\Sigma}_{f_0}(g_1)  \in N^1\ .$$
We proceed by choosing $g_2 \in P^2$ in such a way that $(e^{\gamma_{g_2}} \circ e^{\gamma_{g_1}})_*\gamma_f= (e^{\gamma_{g_2}})_*(( e^{\gamma_{g_1}})_*\gamma_f) =$ $\gamma_{f_0+ f_1^1 + f_2^2 + \ldots }$ with $f_2^2\in N^2$. Continuing in this way, after $r$ steps we obtain that 
$$\Phi:=  e^{\gamma_{g_r}} \circ \ldots \circ e^{\gamma_{g_1}}$$
transforms $\gamma_f$ into $\Phi_*\gamma_f=\gamma_{\overline f}=\gamma_{f_0 + \overline f_1+\ldots}$ where $\overline f_k=f_k^k\in N^k$ for all $1\leq k\leq r$. 

Being the composition of finitely many flows of polynomial coupled cell network vector fields, $\Phi$ is obviously analytic.
\end{proof}
In applications, one is often interested in the bifurcations that occur in the dynamics of a map or differential equation under the variation of external parameters. In the case of coupled cell networks, we may for example assume that $f\in C^{\infty}(V^n\times\R^p,V)$ and let
 $$f^{\lambda}(X):=f(X;\lambda)$$
define a smooth parameter family in $C^{\infty}(V^n,V)$. Correspondingly,  the coupled cell networks $\gamma_{f^{\lambda}}$ form a smooth parameter family in $C^{\infty}(V^N,V^N)$. 

To formulate an appropriate normal form theorem for parameter families of coupled cell networks, we define for $k\geq -1$ and $l\geq 0$,
$$P^{k,l}:=\{f:V^n\times\R^p\to V \ \mbox{homogeneous polynomial of degree} \ k+1\ \mbox{in}\ X\ \mbox{and degree}\ l \ \mbox{in}\ \lambda\}\, .$$ 
We observe that $$[P^{k,l}, P^{K,L}]_{\Sigma}\subset P^{k+K, l+L}\, ,$$
which leads to the following

\begin{theorem}[Coupled cell network normal form theorem with parameters]
 \label{normalformtheoremparameters}
Let $\Sigma=\{\sigma_1, \ldots, \sigma_n\}$ be a semigroup, $f\in C^{\infty}(V^n\times\R^p, V)$ and $f(0;0)=0$. We Taylor expand 
$$f= (f_{-1, 1}+ f_{-1,2} + \ldots ) + (f_{0,0} + f_{0,1}+ f_{0,2} + \ldots) + (f_{1,0} + f_{1,1} + f_{1,2}+\ldots) + \ldots$$ 
with $f_{k,l}\in P^{k,l}$.

Let $1\leq r_1, r_2<\infty$ and for every $-1\leq k \leq r_1$ and $0\leq l\leq r_2$, let $N^{k,l}\subset P^{k,l}$ be a subspace such that 
$$N^{k,l} \oplus \left. {\rm im \ ad}^{\Sigma}_{f_{0,0}} \right|_{P^{k,l}} = P^{k,l}\, .$$
Then there exists a polynomial family $\Phi^{\lambda}$ of analytic diffeomorphisms, defined for $\lambda$ in an open neighborhood of $0$ in $\R^p$ and each sending an open neighborhood of $0$ in $V^N$ to an open neighborhood of $0$ in $V^N$, with the property that $\Phi^{\lambda}$ conjugates $\gamma_{f^{\lambda}}$ to $\gamma_{\overline f^{\lambda}}$, where
$$ \overline f =(\overline f_{-1, 1}+ \overline f_{-1,2} + \ldots ) + (f_{0,0} + \overline f_{0,1}+ \overline f_{0,2} + \ldots) + (\overline f_{1,0} + \overline f_{1,1} + \overline f_{1,2}+\ldots) + \ldots $$
and
$$\overline f_{k,l} \in N^{k,l}  \ \mbox{for all} \ -1\leq k\leq r_1 \ \mbox{and}\ 0\leq l\leq r_2\, .$$
\end{theorem}
\begin{proof}[{\bf Sketch}]
The procedure of normalization is similar as in the proof of Theorem \ref{normalformtheorem}. With respect to ${\rm ad}^{\Sigma}_{f_{0,0}}$, one consecutively normalizes 
\begin{align}
f_{1,0}, f_{2,0}, \ldots, f_{r_1,0}; f_{-1,1}&, f_{0,1}, f_{1,1}, \ldots, f_{r_1,1}; \nonumber 
\\ 
\nonumber 
f_{-1,2}, f_{0,2}, f_{1,2}, \ldots, f_{r_1,2}; \ldots;&  f_{-1,r_2}, f_{0,r_2}, f_{1,r_2}, \ldots, f_{r_1,r_2}\ .
\end{align}
Because $[P^{k,l}, P^{K,L}]_{\Sigma}\subset P^{k+K, l+L}$, we see that once $f_{k,l}$ has been normalized to $\overline f_{k,l}$, it is not changed/affected anymore by any of the subsequent normalization transformations.
\end{proof}
Of course, Theorem \ref{brackettheorem} implies that many other standard results from the theory of normal forms will have a counterpart in the context of semigroup coupled cell networks as well. 

We will compute the normal forms of some network differential equations in Section \ref{examplessection}.

\begin{example}\label{exnormalform}
Again, let $N=3$ and let $\sigma_1, \sigma_2, \sigma_3$ be defined as in Examples \ref{ex1} and \ref{ex2}.
If
\begin{align}\nonumber
& \gamma_f(x_1, x_2, x_3)=(f(x_1, x_1), f(x_2, x_1), f(x_3, x_2))
\end{align} 
is a coupled cell network subject to $\{\sigma_1, \sigma_2\}$, then its normal form will in general be a network subject to $\{\sigma_1, \sigma_2, \sigma_3\}$, i.e. 
\begin{align}\nonumber
&\overline{\gamma_f}(x_1, x_2, x_3)=\gamma_{\overline f}(x_1, x_2, x_3) = (\overline f(x_1, x_1, x_1), \overline f(x_2, x_1, x_1), \overline f(x_3, x_2, x_1))\, .
\nonumber
\end{align}
\end{example}

\section{Symmetry and synchrony}\label{synchronysection}
Symmetry \cite{anto1}, \cite{anto2}, \cite{filipski}, \cite{pivato2}, \cite{golstew4} and synchrony \cite{antonelli2}, \cite{antonelli},\cite{dionne1}, \cite{dionne2}, \cite{romano},  \cite{golstew3}, \cite{golstew2}, \cite{torok}, \cite{stewart1}, \cite{pivato}, \cite{wang} have obtained much attention in the literature on coupled cell networks. They generate and explain interesting patterns, including synchronized states \cite{cortex}, multirythms \cite{pinto2}, \cite{modular}, \cite{pinto}, \cite{parker}, rotating waves \cite{curious}, \cite{golstew3} and synchronized chaos \cite{bursting}, \cite{curious} and can lead to symmetry and synchrony breaking bifurcations, cf. \cite{bifurcations}, \cite{anto4}, \cite{jeroen}, \cite{dias},  \cite{field4}, \cite{field3}, \cite{fieldgolub}, \cite{synbreak2}, \cite{pivato2}, \cite{golstew4}, \cite{perspective}, \cite{parker3}, \cite{synbreak}. In short, symmetry and synchrony heavily impact the dynamics and bifurcations of a network.  

In this section, we relate some of the existing theory on symmetry and synchrony to the semigroup extension that we propose. More precisely, we show that the semigroup extension does not affect the symmetries or synchrony spaces of a network. This implies in particular that the symmetries and synchrony spaces of a network are also present its normal form. The semigroup extension is thus quite harmless and very natural.

To start, let us say that a permutation $p:\{1, \ldots,N\}\to \{1, \ldots, N\}$ of the cells is a {\it network symmetry} for $\Sigma$ if it sends the inputs of a cell to the inputs of its image. That is, if
$$p\circ \sigma_j = \sigma_j\circ p\ \mbox{for all}\ 1\leq j\leq n\, .$$
The permutations with this property obviously form a group. More importantly, they are of dynamical interest because the corresponding representations
$$\lambda_p:V^N\to V^N, \, (x_1, \ldots, x_N)\mapsto (x_{p(1)}, \ldots, x_{p(N)})$$ 
conjugate every coupled cell network map $\gamma_f$ to itself:
\begin{align}
&(\gamma_f\circ \lambda_p)_i(x)=f(\pi_i(x_{p(1)}, \ldots, x_{p(N)}) = f(x_{p(\sigma_1(i))}, \ldots, x_{p(\sigma_n(i)}) = 
\nonumber \\ \nonumber
&f(x_{\sigma_1(p(i))}, \ldots, x_{\sigma_n(p(i)}) = f(\pi_{p(i)}(x))=(\gamma_f(x))_{p(i)}=(\lambda_p\circ \gamma_f)_i(x) \, .
\end{align}
In turn, this means that when $t\mapsto (x_1(t), \ldots, x_N(t))$ is a solution to the differential equations $\dot x=\gamma_f(x)$, then so is $t\mapsto (x_{p(1)}(t), \ldots, x_{p(N)}(t))$. And similarly that when $m\mapsto (x_1^{(m)}, \ldots, x_N^{(m)})$ is an orbit of the map $x^{(m+1)}=\gamma_f(x^{(m)})$, then so is $m\mapsto (x_{p(1)}^{(m)}, \ldots, x_{p(N)}^{(m)})$.

The following lemma states that network symmetries are trivially preserved by our semigroup extension:
\begin{lemma}\label{symlem}
Let $\Sigma=\{\sigma_1, \ldots, \sigma_n\}$ be a collection of maps, not necessarily forming a semigroup, and
$p: \{1, \ldots, N\} \to \{1, \ldots, N\}$ a permutation. 

Then $p$ is a network symmetry for $\Sigma$ if and only if it is a network symmetry for the semigroup $\Sigma'$ generated by $\Sigma$. 
\end{lemma}
\begin{proof}
Elements of the semigroup $\Sigma'$ are of the form $\sigma_{j_1}\circ\ldots\circ \sigma_{j_l}$ for certain $\sigma_{j_k}\in \Sigma$. But if $p\circ \sigma_{j_k}= \sigma_{j_k}\circ p$ for $k=1,\ldots, l$, then also $p \circ (\sigma_{j_1}\circ\ldots\circ \sigma_{j_l})=(\sigma_{j_1}\circ\ldots\circ \sigma_{j_l}) \circ p$. Thus, the collection of network symmetries of $\Sigma$ is the same as the collection of network symmetries of $\Sigma'$.
\end{proof}
Lemma \ref{symlem} implies in particular that the composition $\gamma_f\circ\gamma_g=\gamma_{f\circ_{\Sigma} g}$ and the Lie bracket $[\gamma_f, \gamma_g]=\gamma_{[f,g]_{\Sigma}}$ will exhibit the same network symmetries as $\gamma_f$ and $\gamma_g$.

Though not much more complicated, the situation is slightly more interesting for the synchronous solutions of a network. We recall that a synchrony space of a coupled cell network is an invariant subspace in which certain of the $x_i$ (with $1\leq i \leq N$) are equal. First of all, the following result is classical, see \cite{curious}, \cite{stewart1}.
\begin{proposition} \label{balanced} Let $\Sigma=\{\sigma_1, \ldots, \sigma_n\}$ be a collection of maps, not necessarily forming a semigroup, and
$P=\{P_1,\ldots, P_r \}$ a partition of $\{1,\dots, N\}$. The following are equivalent:
\begin{itemize}
\item[i)] For all $1\leq j\leq n$ and all $1\leq k_1 \leq r$ there exists a $1\leq k_2\leq r$ so that $\sigma_j(P_{k_1})\subset P_{k_2}$.
\item[2)]  For every $f\in C^{\infty}(V^n,V)$ the subspace 
$$ {\rm Syn}_{P} :=\{x\in V^N\, |\ x_{i_1}=x_{i_2} \, \mbox{when} \ i_1 \ \mbox{and}\ i_2\ \mbox{are in the same element of}\ P\, \}$$
is an invariant submanifold for the dynamics of $\gamma_f$.
\end{itemize} 
\end{proposition}
\begin{proof}
The subspace ${\rm Syn}_P$ is invariant under the flow of the differential equation $\dot x=\gamma_f(x)$ if and only if the vector field $\gamma_f$ is tangent to ${\rm Syn}_P$. Similarly, ${\rm Syn}_P$ is invariant under the map $x^{(m+1)}=\gamma_f(x^{(m)})$ if and only if $\gamma_f$ sends ${\rm Syn}_P$ to itself. Both properties just mean that for all $x\in {\rm Syn}_P$ it holds that
$$f(x_{\sigma_1(i_1)}, \ldots, x_{\sigma_n(i_1)}) = f(x_{\sigma_1(i_2)}, \ldots, x_{\sigma_n(i_2)})\ \mbox{for all}\ i_1, i_2 \ \mbox{in the same element of} \ P\, .$$ 
The latter statement holds for all $f\in C^{\infty}(V^n,V)$ if and only if for all $x\in {\rm Syn}_P$,
$$x_{\sigma_j(i_1)} = x_{\sigma_j(i_2)} \ \mbox{for all}\  1\leq j \leq n\ \mbox{and all}\ i_1,i_2\ \mbox{in the same element of}\ P\, .$$
It is not hard to see that this is true precisely when all $\sigma_j\in \Sigma$ map the elements of $P$ into elements of $P$.
\end{proof}
A partition $P$ of $\{1,\ldots,N\}$ with property {\it i)} is sometimes called a {\it balanced partition} or {\it balanced coloring} and a subspace ${\rm Syn}_P$ satisfying property {\it ii)} a {\it (robust) synchrony space}. 

The following result says that the synchrony spaces of a network do not change if one extends the network architecture to a semigroup:
\begin{lemma}\label{robustness}
Let $\Sigma=\{\sigma_1, \ldots, \sigma_n\}$ be a collection of maps, not necessarily forming a semigroup, and
$P=\{P_1,\ldots, P_r\}$ a partition of $\{1,\dots, N\}$.

Then ${\rm Syn}_P$ is a (robust) synchrony space for $\Sigma$ if and only if it is a (robust) synchrony space for the semigroup $\Sigma'$ generated by $\Sigma$. 
\end{lemma}
\begin{proof}
Elements of the semigroup $\Sigma'$ are of the form $\sigma_{j_1}\circ\ldots\circ \sigma_{j_l}$ for certain $\sigma_{j_k}\in \Sigma$. This implies that the elements of $\Sigma$ send the elements of $P$ inside elements of $P$ if and only if the elements of $\Sigma'$ do. In other words: that the collection of balanced partitions of $\Sigma$ and of $\Sigma'$ are the same. The result now follows from Proposition \ref{balanced}.
\end{proof}
Lemma \ref{robustness} implies in particular that the composition $\gamma_f\circ\gamma_g=\gamma_{f\circ_{\Sigma}g}$ and the Lie bracket $[\gamma_f, \gamma_g]=\gamma_{[f,g]_{\Sigma}}$ will exhibit the same synchrony spaces as $\gamma_f$ and $\gamma_g$.

We conclude this section with the following simple but important observation:
\begin{corollary}
Let $\Sigma =\{\sigma_1, \ldots, \sigma_n\}$ be a collection of maps, not necessarily forming a semigroup, and $\gamma_f$ a coupled cell network vector field subject to $\Sigma$. 

Then a local normal form $\gamma_{\overline f}$ of $\gamma_f$ has the same network symmetries and the same synchrony spaces as $\gamma_f$.
\end{corollary}
\begin{proof}
$\gamma_{\overline f}$ is a coupled cell network with respect to the semigroup $\Sigma'$ generated by $\Sigma$. Thus, the result follows from Lemma \ref{symlem} and Lemma \ref{robustness}.
\end{proof}
\begin{example}
Again, let $N=3$ and let $\sigma_1, \sigma_2, \sigma_3$ be defined as in Examples \ref{ex1} and \ref{ex2}. 
Recall that a coupled cell network differential equation subject to $\{\sigma_1, \sigma_2\}$ is of the form
\begin{align}\nonumber
\dot x_1 = f(x_1, x_1), \, \dot x_2=f(x_2, x_1),\,  \dot x_3 = f(x_3, x_2)\, .
\end{align}
These equations do not have any network symmetries, but they do admit the nontrivial balanced partitions 
$$\{1,2\}\cup\{3\}\ \mbox{and} \ \{1,2,3\}\, .$$ 
The corresponding invariant synchrony spaces $\{ x_1=x_2\}$ and $\{x_1=x_2=x_3\}$ are preserved in the normal form, because the latter is a coupled cell network subject to $\{\sigma_1, \sigma_2, \sigma_3\}$.
\end{example}

\section{Input symmetries}\label{symmetrysection}
In several of the coupled cell networks that appear in the mathematical literature, the function $f$ that determines the network dynamics is assumed invariant under the permutation of some of its inputs. In this section, we point out some conditions under which such an {\it input symmetry} can be preserved in the normal form of $f$. This section is not important for the remainder of this paper and can be skipped at first reading. 

An input symmetry is reflected by a permutation
$q:\{1, \ldots, n\}\to\{1, \ldots, n\}$ with the property that $f\circ \lambda_q \circ \pi_i=f\circ \pi_i$ for all $1\leq i\leq N$. That is, for which
\begin{align}\label{invariance}
f(x_{\sigma_{1}(i)},\ldots, x_{\sigma_n(i)})= f(x_{\sigma_{q(1)}(i)}, \ldots, x_{\sigma_{q(n)}(i)})\ \mbox{for all}\ x\in V^N \ \mbox{and all} \ 1\leq i\leq N\, .
\end{align}
The input symmetries of $f$ obviously form a group.

We do not aim here to give a full answer to the question under which conditions the invariance (\ref{invariance}) can be preserved in a normal form, because this question is very delicate. The first problem is that the presence of an input symmetry makes that the maps $\sigma_1, \ldots, \sigma_n$ are not uniquely determined. A second complication is that an input symmetry of $f$ has a nontrivial impact on its robust synchrony spaces.

Instead, we will only consider the case here that an input symmetry gives rise to a network symmetry. This means that along with the permutation $q$ of the inputs $\{1, \ldots, n\}$ there exists a permutation $p$ of the cells $\{1, \ldots, N\}$ that sends the $j$-th input of each cell to the $q(j)$-th input of its image, i.e. that
\begin{align}\label{inputsymmetry}
p\circ \sigma_j = \sigma_{q(j)} \circ p \ \mbox{for all} \ 1\leq j \leq n\, .
\end{align}
We call a permutation $q$ for which there exists a permutation $p$ so that (\ref{invariance}) and (\ref{inputsymmetry}) hold a {\it dynamical input symmetry}. We remark that if
$p_1\circ \sigma_j=\sigma_{q_1(j)}\circ p_1$ and $p_2\circ \sigma_j=\sigma_{q_2(j)}\circ p_2$ for all $1\leq j\leq n$, then 
$$(p_1\circ p_2) \circ \sigma_j = \sigma_{(q_1\circ q_2)(j)}\circ(p_1\circ p_2)\, .$$  
This implies that the dynamical input symmetries form a subgroup of the group of all input symmetries. They are precisely the input symmetries that correspond to a dynamical symmetry of the network:
\begin{proposition}
Let $p$ be a permutation of $\{1,\dots, N\}$ and $q$ a permutation of $\{1,\ldots, n\}$. Assume that 
$p\circ \sigma_j = \sigma_{q(j)} \circ p$ for all $1\leq j \leq n$ and that (\ref{invariance}) holds.
Then  
$$\gamma_f\circ \lambda_p= \lambda_p\circ \gamma_f\, .$$ 
\end{proposition}

\begin{proof}
First of all we claim that
\begin{align}\label{pipi}
\mbox{when}\ p\circ \sigma_j = \sigma_{q(j)} \circ p \ \mbox{for all} \ 1\leq j \leq n\ \mbox{then}\ \pi_i\circ \lambda_p = \lambda_q \circ \pi_{p(i)}\ \mbox{for all}\ 1\leq i\leq N\, .
\end{align} 
This follows from a little computation:
\begin{align}
(\pi_i\circ & \lambda_p)(x) = \pi_i(x_{p(1)}, \ldots, x_{p(N)}) = (x_{p(\sigma_1(i))}, \ldots, x_{p(\sigma_n(i))})= \nonumber \\ \nonumber
(x_{\sigma_{q(1)}(p(i))}&, \ldots, x_{\sigma_{q(n)}(p(i))}) = \lambda_q(x_{\sigma_{1}(p(i))}, \ldots, x_{\sigma_{n}(p(i))}) = (\lambda_q\circ \pi_{p(i)})(x)\, .
\end{align}
Using (\ref{pipi}) and our assumption (\ref{invariance}) that $f\circ \lambda_q=f$ on every ${\rm im}\, \pi_i$, we hence find that
\begin{align}\nonumber
(\gamma_f\circ \lambda_p)_i= f \circ \pi_i\circ \lambda_p = f\circ \lambda_q\circ \pi_{p(i)} = f\circ \pi_{p(i)} =(\gamma_f)_{p(i)} = (\lambda_p\circ \gamma_f)_i\, .
\end{align}
\end{proof}
\noindent The relevance of dynamical input symmetries for normal forms is explained below. We first show that dynamical input symmetries survive the semigroup extension.

\begin{lemma}
Let $\Sigma=\{\sigma_1, \ldots, \sigma_n\}$ be a collection of maps, not necessarily forming a semigroup and let $q:\{1, \ldots, n\}\to\{1, \ldots, n\}$ be a dynamical input symmetry for $\Sigma$. 

Then the latter extends to a unique dynamical input symmetry 
$$q':\{1, \ldots, n, n+1, \ldots, n'\}\to \{1, \ldots, n, n+1, \ldots, n'\}$$ for the semigroup $\Sigma'=\{\sigma_1, \ldots, \sigma_n, \sigma_{n+1}, \ldots, \sigma_{n'}\}$ generated by $\Sigma$.
\end{lemma}
\begin{proof}
Recall that elements of $\Sigma'$ are of the form $\sigma_{j_1}\circ\ldots \circ \sigma_{j_l}$ for certain $\sigma_{j_k}\in\Sigma$. Assume now that $p\circ \sigma_{j_1}=\sigma_{q(j_1)}\circ p$ and $p\circ \sigma_{j_2}=\sigma_{q(j_2)}\circ p$. 
Then it follows that 
$$p\circ \sigma_{\widetilde \sigma_{j_1}(j_2)} = p \circ (\sigma_{j_1} \circ \sigma_{j_2})= \sigma_{q(j_1)} \circ p \circ \sigma_{j_2}  = (\sigma_{q(j_1)}\circ \sigma_{q(j_2)}) \circ p = \sigma_{\widetilde \sigma_{q(j_1)}(q(j_2))}\circ p\, .$$ 
This means that if an extension $q': \{1, \ldots, n, n+1, \ldots, n'\}\to \{1, \ldots, n, n+1, \ldots, n'\} $ exists, then it must be unique and satisfy
$$q'(\widetilde \sigma_{j_1}(j_2))= \widetilde \sigma_{q(j_1)}(q(j_2))\, .$$ 
If now $p\circ \sigma_{j_3}=\sigma_{q(j_3)}\circ p$ and $p\circ \sigma_{j_4}=\sigma_{q(j_4)}\circ p$ and $\widetilde \sigma_{j_1}(j_2)=\widetilde \sigma_{j_3}(j_4)$, then actually
$$\sigma_{\widetilde \sigma_{q(j_1)}(q(j_2))}\circ p = p\circ \sigma_{\widetilde \sigma_{j_1}(j_2)} = p\circ \sigma_{\widetilde \sigma_{j_3}(j_4)} =\sigma_{\widetilde \sigma_{q(j_3)}(q(j_4))}\circ p\, .$$
From this it follows that $\widetilde \sigma_{q(j_1)}(q(j_2))= \widetilde \sigma_{q(j_3)}(q(j_4))$ if $\widetilde \sigma_{j_1}(j_2)=\widetilde \sigma_{j_3}(j_4)$ and hence that $q'$ is well-defined.
\end{proof}
The following result explains that network symmetries are preserved under taking compositions and Lie brackets:
\begin{theorem}\label{invariancetheorem}
Let $\Sigma$ be a semigroup and assume $p\circ \sigma_j = \sigma_{q(j)} \circ p$ for all $1\leq j \leq n$. Then
\begin{align}
\begin{array}{rll} 
(f\circ_{\Sigma}g)\circ \lambda_q = &\!\!\!\! (f\circ \lambda_q)\circ_{\Sigma}(g\circ \lambda_q) & \!\! \mbox{on every}\ \, {\rm im}\, \pi_i\, ,
\\
\ [f,g]_{\Sigma}\circ \lambda_q  = & \!\!\!\! [ f\circ \lambda_q,g\circ \lambda_q]_{\Sigma} & \!\! \mbox{on every}\ \, {\rm im}\, \pi_i\, .
\end{array}
\end{align}
\end{theorem}

\begin{proof} With slight abuse of notation, we write
\begin{align}\nonumber
&(f \circ_{\Sigma} g)\circ \lambda_q\circ\pi_{p(i)} = f(\ldots, g(A_{\sigma_j}\circ \lambda_q \circ \pi_{p(i)}), \ldots)= f(\ldots, g(A_{\sigma_j}\circ\pi_i\circ \lambda_p), \ldots)=\\ \nonumber &
f(\ldots, g(\pi_{\sigma_j(i)}\circ \lambda_p), \ldots) = f(\ldots, g(\lambda_{q}\circ\pi_{p(\sigma_j(i))}), \ldots)= f(\ldots, (g\circ \lambda_q)(\pi_{p(\sigma_j(i))}), \ldots)= \\
\nonumber &f(\ldots, (g\circ \lambda_q)(\pi_{\sigma_{q(j)}(p(i))}), \ldots) = (f\circ \lambda_q)(\ldots, (g\circ \lambda_q)(\pi_{\sigma_{j}(p(i))}), \ldots)=\\
\nonumber &(f\circ \lambda_q)(\ldots, (g\circ \lambda_q)(A_{\sigma_j}\circ \pi_{p(i)}), \ldots) =((f\circ \lambda_q)\circ_{\Sigma}(g\circ \lambda_q))\circ \pi_{p(i)}\, .
\end{align}
This proves that $(f\circ_{\Sigma}g)\circ \lambda_q = (f\circ \lambda_q)\circ_{\Sigma}(g\circ \lambda_q)$ on each ${\rm im}\, \pi_{p(i)}$ and hence, because $p$ is invertible, on each ${\rm im}\, \pi_i$. The proof for the Lie bracket is similar.
\end{proof}
\noindent Finally, we conclude
\begin{corollary}
Let $\Sigma =\{\sigma_1, \ldots, \sigma_n\}$ be a collection of maps, not necessarily forming a semigroup, and $\gamma_f$ a coupled cell network vector field subject to $\Sigma$. 

Then the local normal form $\gamma_{\overline f}$ of $\gamma_f$ can be chosen to have the same dynamical input symmetries as $\gamma_f$.
\end{corollary}
\begin{proof}
Let $G$ denote the group of dynamical input symmetries of $f = f_0+f_1+f_2+\ldots$ and let us define the set of $G$-invariant functions
$$C^{\infty}_G(V^n,V):= \{g\in C^{\infty}(V^n, V)\, | \, g\circ \lambda_q\circ \pi_i= g\circ\pi_i \ \mbox{for all}\ q\in G\  \mbox{and all}\ 1\leq i\leq N\}\, .$$
Theorem \ref{invariancetheorem} implies that if $g, h\in C^{\infty}_G(V^n,V)$, then also $$[g,h]_{\Sigma} \in C^{\infty}_G(V^n,V)\, .$$
The fact that $f\in C^{\infty}_G(V^n, V)$ moreover implies that
$$f_k\in P^k_G:=\{g_k\in P^k\, | \, g_k\circ \lambda_q \circ \pi_i= g_k \circ\pi_i\ \mbox{for all}\ q\in G\ \mbox{and all}\ 1\leq i\leq N\} $$ 
is a $G$-invariant polynomial of degree $k+1$. 

It clearly holds that
$[P^k_G, P^l_G]_{\Sigma}\subset P^{k+l}_G$. As a consequence, we can repeat the proof of Theorem \ref{normalformtheorem} by replacing $P^k$ by $P^k_G$ and choosing the normal form spaces $N^k_G\subset P^k_G$ so that 
$$\left.{\rm im\ ad}_{f_0}\right|_{P^k_G} \oplus N^k_G = P^k_G\, .$$ 
This produces a normal form $\overline f \in C^{\infty}_G(V^n,V)$.
\end{proof}

\begin{example}
 \label{funnyexample}
Consider the class of differential equations of the form
 \begin{align}\nonumber
\begin{array}{ll} 
\dot x_1 =& f(x_1,x_2,x_2,x_1)\, , \\
\dot x_2 =& f(x_1,x_2,x_1,x_2)\, .
\end{array}
\end{align}
These differential equations have a semigroup coupled cell network structure with $N=2$ and $n=4$, see Figure \ref{funnypicture}.
 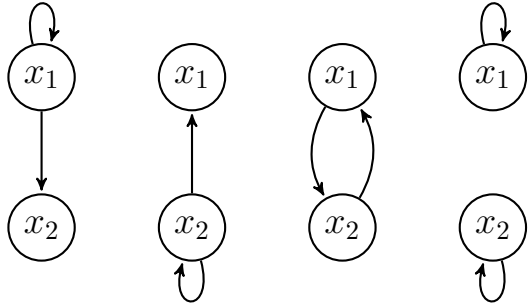
\begin{figure}[ht]\renewcommand{\figurename}{\rm \bf \footnotesize Figure}
\centering
\begin{tikzpicture}[->,>=stealth',shorten >=1pt,auto,node distance=2cm,
                    thick,main node/.style={circle,draw,font=\sffamily\Large\bfseries}]

  \node[main node] (1) {$x_1$};
  \node[main node] (2) [below of=1] {$x_2$};
  \node[main node] (3) [right of=1] {$x_1$};
  \node[main node] (4) [right of=2] {$x_2$};
  \node[main node] (5) [right of=3] {$x_1$};
  \node[main node] (6) [right of=4] {$x_2$};
\node[main node] (7) [right of=5] {$x_1$};
  \node[main node] (8) [right of=6] {$x_2$};
  \path[every node/.style={font=\sffamily\small}]
    (1) edge node {} (2)
 (1) edge [loop above] node {} (1)
(4) edge node {} (3)
 (4) edge [loop below] node {} (4)   
(5) edge  [bend right] node {} (6)
(6) edge  [bend right] node {} (5)
    (7) edge [loop above] node {} (7)
 (8) edge [loop below] node {} (8);
\end{tikzpicture}

\caption{\footnotesize {\rm The network with $N=2$ and $n=4$.}}
\label{funnypicture}
\end{figure}

\noindent The semigroup $\Sigma$ in this case is the full non-Abelian semigroup of maps on $2$ symbols. In other words, $\Sigma=\{\sigma_1, \sigma_2, \sigma_3, \sigma_4\}$, where 
$$\sigma_1(1) = 1, \sigma_1(2) = 1, \sigma_2(1) = 2, \sigma_2(2) = 2, \sigma_3(1) = 2, \sigma_3(2) = 1, \sigma_4(1) = 1, \sigma_4(2) = 2\, . $$
There is only one nontrivial permutation $p:\{1,2\}\to\{1,2\}$ of the cells, which is defined by $p(1):=2$ and $p(2):=1$. It is easily checked that 
$$p\circ \sigma_1 = \sigma_2\circ p,\ p\circ \sigma_2 = \sigma_1\circ p,\ p\circ \sigma_3 = \sigma_3\circ p,\ p\circ \sigma_4 = \sigma_4\circ p\, .$$
In other words, $p\circ \sigma_j=\sigma_{q(j)}\circ p$ if we let $q:\{1,2,3,4\}\to\{1,2,3,4\}$ be defined by 
$$q(1)=2, q(2)=1, q(3)=3, q(4)=4\, .$$ 
Thus, the (in this case identical) invariances 
$$f(x_1,x_2, x_2,x_1)=f(x_2,x_1,x_2,x_1), \ f(x_1,x_2, x_1,x_2)=f(x_2,x_1,x_1,x_2)$$
can be preserved in the normal form of $f$. These are precisely the invariances that make $\lambda_p:(x_1, x_2)\mapsto (x_2, x_1)$ a symmetry  of the differential equations.
\end{example}

\section{SN-decomposition}\label{SNsection}
We recall from the previous sections that when $f_0\in L(V^n,V)$, then ${\rm ad}^{\Sigma}_{f_0}:P^k\to P^k$. The operators ${\rm ad}^{\Sigma}_{f_0}|_{P^k}$ are called ``homological operators'' and they play an important role in normal form theory. This is first of all because the normal form spaces $N^k\subset P^k$ of Theorem \ref{normalformtheorem} must be chosen complementary to their images, and secondly because in computing a normal form one needs to ``invert'' them when solving the homological equations ${\rm ad}^{\Sigma}_{f_0}(g_k)-h_k\in N^k$, see the proof of Theorem \ref{normalformtheorem}. 

For this reason, it is convenient to have at one's disposal the ``SN-decompositions'' (also called ``Jordan-Chevalley decompositions'') of the homological operators \cite{sanderscushman}, \cite{murdock}, \cite{sanvermur}. We recall that, since $P^k$ is finite-dimensional, the map ${\rm ad}^{\Sigma}_{f_0}|_{P^k}$ admits a unique SN-decomposition
\begin{align}\label{SN}
{\rm ad}^{\Sigma}_{f_0}|_{P^k}=({\rm ad}^{\Sigma}_{f_0}|_{P^k})^S+({\rm ad}^{\Sigma}_{f_0}|_{P^k})^N\, 
\end{align}
in which the map $({\rm ad}^{\Sigma}_{f_0}|_{P^k})^S$ is semisimple, the map $({\rm ad}^{\Sigma}_{f_0}|_{P^k})^N$ is nilpotent and the two maps $({\rm ad}^{\Sigma}_{f_0}|_{P^k})^S$ and $({\rm ad}^{\Sigma}_{f_0}|_{P^k})^N$ commute. The aim of this somewhat technical section is to characterize this SN-decomposition in an as simple as possible way. We will do this in a number of steps, starting from the following technical result:
\begin{proposition}\label{injectivity}
Assume that $\Sigma$ is a semigroup. Then the linear map 
$$\gamma|_{L(V^n, V)}: L(V^n, V)\to L(V^N, V^N)\ \mbox{sending}\ f_0\ \mbox{to}\ \gamma_{f_0}\ \mbox{is injective}\ .$$
\end{proposition}
\begin{proof}
The definition $(\gamma_{f_0})_i:=f_0\circ\pi_i$ implies that when $f_0$ is linear, then $\gamma_{f_0}=0$ precisely when $f_0$ vanishes on ${\rm im}\, \pi_1 + \ldots + {\rm im}\, \pi_N$. Thus, all we need to show is that $${\rm im}\ \pi_1 + \ldots + {\rm im}\ \pi_N = V^n\, .$$
Equivalently, we will show by contradiction that the map
$$`` \pi_1+ \ldots+ \pi_N " : (V^{N})^ N \to V^{n}\ , \ (x^{(1)}, \ldots, x^{(N)})\mapsto \pi_1(x^{(1)})+\ldots + \pi_N(x^{(N)})$$
must be surjective. To this end, let us define for $1\leq j\leq n$ the maps
$$\pi^j:(V^{N})^N\to V \ \mbox{by}\ \pi^j(x^{(1)}, \ldots, x^{(N)}) = (\pi_1(x^{(1)})+ \ldots+ \pi_N(x^{(N)}))_j = x^{(1)}_{\sigma_j(1)}+ \ldots+ x^{(N)}_{\sigma_j(N)}\, .$$
In other words, $\pi^j$ is $``\pi_1+\ldots+\pi_N"$ followed by the projection to the $j$-th factor of $V^n$. In particular, if $``\pi_1+\ldots+\pi_N"$ is not surjective, then there is a relation of the form 
$$\pi^j=\sum_{k\neq j}^n\lambda_k\pi^k\, .$$
This means that 
$$ x^{(1)}_{\sigma_j(1)}+ \ldots+ x^{(N)}_{\sigma_j(N)} = \sum_{k\neq j} \lambda_k\left( x^{(1)}_{\sigma_k(1)}+ \ldots+ x^{(N)}_{\sigma_k(N)}\right) \ \mbox{for all}\ x^{(1)}, \ldots, x^{(N)}\in V^N\, .$$
It is clear that this can only be true if $\sigma_j= \sigma_k$ for some $k\neq j$. This is a contradiction, because we assumed that the elements of $\Sigma$ are distinct.
\end{proof}
Next, we recall that when $f_0\in L(V^n, V)$, then $\gamma_{f_0}\in L(V^N, V^N)$ and thus also the latter has an SN-decomposition
$$\gamma_{f_0}=\gamma_{f_0}^S+\gamma_{f_0}^N$$ 
for certain $\gamma_{f_0}^S, \gamma_{f_0}^N\in L(V^N,V^N)$. We are going to relate the SN-decomposition of $\gamma_{f_0}$ to that of ${\rm ad}^{\Sigma}_{f_0}|_{P^k}$. But first we show that both $\gamma_{f_0}^S$ and $\gamma_{f_0}^N$ are coupled cell network maps:
\begin{lemma}\label{propsndecomposition}
For every $f_0\in L(V^n, V)$ there exist unique $f_0^S, f_0^N\in L(V^n,V)$ so that
$$\gamma_{f_0}^S=\gamma_{f_0^S}\ \mbox{and} \ \gamma_{f_0}^N=\gamma_{f_0^N}\, .$$
It holds that $f_0=f_0^S+f_0^N$ and that $[f_0^S, f_0^N]_{\Sigma}=0$. In particular ${\rm ad}^{\Sigma}_{f_0^S} \circ {\rm ad}^{\Sigma}_{f_0^N} = {\rm ad}^{\Sigma}_{f_0^N} \circ {\rm ad}^{\Sigma}_{f_0^S}$.
\end{lemma}

\begin{proof}
We recall - see for instance \cite{humphreys}, pp. 17 - that both the semisimple part $\gamma_{f_0}^S$ and the nilpotent part $\gamma_{f_0}^N$ of $\gamma_{f_0}$ are polynomial functions of $\gamma_{f_0}$. More precisely, $\gamma_{f_0}^{S}=p(\gamma_{f_0})$ and $\gamma_{f_0}^N=\gamma_{f_0}-p(\gamma_{f_0})$, where 
$$p(\gamma)= a_0I + a_1\gamma + \ldots + a_d \gamma^d$$ 
is a polynomial with coefficients $a_0, \ldots, a_d\in\C$.

Theorem \ref{composition} then implies that $\gamma_{f_0}^S=p(\gamma_{f_0}) = \gamma_{p(f_0)}= \gamma_{f_0^S}$ for $f_0^S\in L(V^n,V)$ defined as
\begin{align}\label{polynomial}
f_0^S=p(f_0)=a_0f_0 +a_1(f_0\circ_{\Sigma} f_0) + \ldots + a_d (f_0\circ_{\Sigma} \ldots \circ_{\Sigma} f_0) \ .
\end{align} 
By Lemma \ref{associative} this $f_0^S$ is well-defined, while by Proposition \ref{injectivity} it is unique. Similarly, $\gamma_{f_0}^N=\gamma_{f_0^N}$ for a well-defined and unique $f_0^N= f_0-p(f_0)\in L(V^n,V)$. Clearly, $f_0=f_0^S+f_0^N$.

Because $[f_0^S, f_0^N]_{\Sigma} \in L(V^n, V)$ and 
$$\gamma_{[f_{0}^S, f_{0}^N]_{\Sigma} } = [\gamma_{f_0^S}, \gamma_{f_0^N}] = [\gamma_{f_0}^S, \gamma_{f_0}^N] =  0$$
it follows from Proposition \ref{injectivity} that $[f_{0}^S, f_0^N]_{\Sigma}=0$. The Jacobi identity 
$$[{\rm ad}^{\Sigma}_{f_0^S}, {\rm ad}^{\Sigma}_{f_0^N}] = {\rm ad}^{\Sigma}_{[f_0^S,f_0^N]_{\Sigma}}$$ 
then confirms that ${\rm ad}^{\Sigma}_{f_0^S}$ and ${\rm ad}^{\Sigma}_{f_0^N}$ commute as operators on $C^{\infty}(V^n, V)$. 
\end{proof}

\noindent Before we come to the desired characterization of the SN-decomposition of ${\rm ad}_{f_0}^{\Sigma}|_{P^k}$, we need to make one simple observation. It concerns the fact that two functions $f, g\in C^{\infty}(V^n, V)$ generate the same network map (in the sense that $\gamma_f=\gamma_g$) if and only if $f-g\in \ker \gamma =\{h\, | \, \gamma_h=0\} \subset C^{\infty}(V^n, V)$. Thus, $\ker \gamma$ consists of those functions $h:V^n\to V$ that are irrelevant for the dynamics of coupled cell networks.
  
One can remark that when $f, g\in C^{\infty}(V^n, V)$ and $f\in \ker \gamma$, then $\gamma_{[f,g]_{\Sigma}}=[\gamma_f, \gamma_g]=0$ and hence also $[f, g]_{\Sigma}\in\ker \gamma$. This means $\ker \gamma\subset C^{\infty}(V^n, V)$ is a Lie algebra ideal. In particular it holds for every $f\in C^{\infty}(V^n, V)$ that the adjoint map ${\rm ad}^{\Sigma}_f: C^{\infty}(V^n, V)\to C^{\infty}(V^n, V)$ sends $\ker \gamma$ to $\ker \gamma$ and hence that ${\rm ad}_{f}^{\Sigma}$ descends to a well-defined map on $C^{\infty}(V^n, V)/\ker \gamma$. We can now formulate our result:

\begin{theorem}\label{SNtheorem}
For every $k=0,1,2,\ldots$ the maps
$$ ({\rm ad}^{\Sigma}_{f_0}|_{P^k})^S \ \mbox{and}\ \left. {\rm ad}^{\Sigma}_{f_0^S}\right|_{P^k} \ \mbox{respectively}\  ({\rm ad}^{\Sigma}_{f_0}|_{P^k})^N \ \mbox{and}\ \left.{\rm ad}^{\Sigma}_{f_0^N}\right|_{P^k}$$
descend to the same map on $P^k /\ker \gamma$. 
\end{theorem}

\begin{proof} We start by repeating that ${\rm ad}_{f_0}^{\Sigma}$ maps $\ker \gamma$ into itself and hence descends to a well-defined map on $C^{\infty}(V^n, V)/\ker \gamma$ that moreover sends $P^k/\ker \gamma$ into itself. More interestingly, since $\left({\rm ad}_{f_0}^{\Sigma}|_{P^k}\right)^S$ and $\left({\rm ad}_{f_0}^{\Sigma}|_{P^k}\right)^N$ are polynomial functions of ${\rm ad}_{f_0}^{\Sigma}|_{P^k}$, also these latter maps send $\ker \gamma$ into itself and thus descend to $P^k/\ker \gamma$.

For the actual proof of the theorem, we define for $k\geq 0$ the vector spaces
$$\mathcal{P}^k:=\{\mbox{homogeneous polynomial vector fields on}\ V^N\ \mbox{of degree}\ k+1\}\, .$$
It is clear that $\gamma: P^k/\ker \gamma \to \mathcal{P}^k$ is an injective linear map. Moreover, the computation 
$$(\gamma\circ {\rm ad}^{\Sigma}_{f_0})(g_k) = \gamma_{[f_0, g_k]_{\Sigma}} = [\gamma_{f_0}, \gamma_{g_k}]= ({\rm ad}_{\gamma_{f_0}}\circ \gamma)(g_k)$$
reveals that the maps ${\rm ad}^{\Sigma}_{f_0}: P^k/\ker \gamma\to P^k/\ker \gamma$ and ${\rm ad}_{\gamma_{f_0}}:\mathcal{P}^k\to \mathcal{P}^k$ are conjugate by the map $\gamma: P^k/\ker \gamma \to\mathcal{P}^k$. Similarly, ${\rm ad}^{\Sigma}_{f_0^S}$ is conjugate to ${\rm ad}_{\gamma_{f_0^S}}$ and
${\rm ad}^{\Sigma}_{f_0^N}$ is conjugate to ${\rm ad}_{\gamma_{f_0^N}}$. Now we recall the well-known fact that the SN-decomposition of ${\rm ad}_{\gamma_{f_0}}|_{\mathcal{P}^k}$ is
$${\rm ad}_{\gamma_{f_0}}|_{\mathcal{P}^k} = \left({\rm ad}_{\gamma_{f_0}}|_{\mathcal{P}^k}\right)^S + \left({\rm ad}_{\gamma_{f_0}}|_{\mathcal{P}^k}\right)^N 
 = \left.{\rm ad}_{\gamma_{f_0}^S}\right|_{\mathcal{P}^k}  + \left. {\rm ad}_{\gamma_{f_0}^N}\right|_{\mathcal{P}^k} =
 \left.{\rm ad}_{\gamma_{f_0^S}}\right|_{\mathcal{P}^k} + \left.{\rm ad}_{\gamma_{f_0^N}}\right|_{\mathcal{P}^k}\, .$$ 
Because $\gamma$ is injective, we have thus proved that 
$$\left.{\rm ad}^{\Sigma}_{f_0}\right|_{P^k} = \left. {\rm ad}^{\Sigma}_{f_0^S}\right|_{P^k}  + \left.{\rm ad}^{\Sigma}_{f_0^N}\right|_{P^k} : P^k/\ker \gamma\to P^k/\ker \gamma$$
is the SN-decomposition of the quotient map. Because the SN-decomposition of the quotient is the quotient of the SN-decomposition, this proves the theorem.
\end{proof}
\noindent Because the elements of $\ker \gamma$ are dynamically completely irrelevant, for all practical purposes we can think of Theorem \ref{SNtheorem} as saying that 
$$``\, \left.{\rm ad}^{\Sigma}_{f_0}\right|_{P^k} = \left.{\rm ad}^{\Sigma}_{f_0^S}\right|_{P^k} + \left.{\rm ad}^{\Sigma}_{f_0^N}\right|_{P^k}\ \mbox{is the SN-decomposition of}\ \left.{\rm ad}_{f_0}^{\Sigma}\right|_{P^k}\, "\, .$$ 
This is very convenient, because it means that one can determine the SN-decompositions of all operators $\left.{\rm ad}_{f_0}^{\Sigma}\right|_{P^k}$ simultaneously by simply determining the splitting $f_0=f_0^S+f_0^N$.
 
\begin{example} Even though by Proposition \ref{injectivity} the restriction $\gamma|_{L(V^n, V)}$ is injective, the full map $\gamma:C^{\infty}(V^n, V)\to C^{\infty}(V^N, V^N)$ may fail to be so. This situation occurs when $\cup_{i=1}^N{\rm im}\, \pi_i \neq V^n$, because $\gamma_f=0$ already when $f$ vanishes on every ${\rm im}\, \pi_i$. This is the reason for the somewhat difficult formulation of Theorem \ref{SNtheorem}.

To illustrate this phenomenon, we refer to Example \ref{funnyexample} in which 
$$\pi_1(x_1,x_2)=(x_1,x_2,x_2,x_1) \ \mbox{and}\ \pi_2(x_1,x_2)=(x_1,x_2,x_1,x_2)\, ,$$ 
so that in particular ${\rm im}\, \pi_1 \cup {\rm im} \, \pi_2 \neq V^4$. When for instance $V=\R$ in this example, then $\ker \gamma\subset C^{\infty}(\R^4, \R)$ is the ideal generated by
$$ (X_1- X_4)(X_1-X_3), (X_1-X_4)(X_2-X_4), (X_2-X_3)(X_1-X_3)\ \mbox{and} \ (X_2-X_3)(X_2-X_4)\ .$$
\end{example}

\noindent We conclude this section with the following dynamical implication of Theorem \ref{SNtheorem}:
\begin{corollary}
Let $0\leq r<\infty$. Then it can be arranged that the normal form $\overline f=f_0+\overline f_1+\ldots \in C^{\infty}(V^n, V)$ of an $f=f_0+\overline f_1+\ldots  \in C^{\infty}(V^n, V)$ has the special property that the truncated normal form coupled cell map/vector field
$\gamma_{f_0+\overline f_1 + \ldots + \overline f_r}$
commutes with the continuous family of maps $$t\mapsto e^{t\gamma_{f_0^S}}\, .$$
\end{corollary}
\begin{proof} 
Recall that each one of the normal form spaces $N^k \subset P^k$ of Theorem \ref{normalformtheorem} is required to have the property that $N^k\oplus\left. {\rm im \ ad}^{\Sigma}_{f_0}\right|_{P^k} = P^k$. It is not hard to see that whenever
\begin{align}\label{kerim}
N^k\subset \ker\, ({\rm ad}^{\Sigma}_{f_0}|_{P^k})^S\  \mbox{is complementary to}\,  \ {\rm im\, (ad}^{\Sigma}_{f_0}|_{P^k})^N
\end{align}
then this condition is fulfilled. Thus, let us choose $N^k$ so that it satisfies (\ref{kerim}). The fact that $\left({\rm ad}_{f_0^S}|_{P^k}\right)^S$ and $\left.{\rm ad}_{f_0^S}\right|_{P^k}$ descend to the same map on $P^k/\ker \gamma$ implies in particular that for such $N^k$ it holds that ${\rm ad}_{f_0^S}(N^k)\subset \ker \gamma$.

Let now $\overline f= f_0+\overline f_1+ \overline f_2 + \ldots$ be any normal form of $f$ of order $r$ with respect to the $N^k$, meaning that $\overline f_k\in N^k$ for all $1\leq k\leq r$. Such a normal form exists by Theorem \ref{normalformtheorem}.
Then it holds that ${\rm ad}^{\Sigma}_{f_0^S}(\overline f_k)\in \ker \gamma$ and in view of Theorem \ref{composition} we therefore have 
\begin{align}\nonumber
[\gamma_{f_0^S}, \gamma_{f_0+ \overline f_1 \ldots +\overline f_r}]= 
\gamma_{[f_0^S, f_0+ \overline f_1 + \ldots+\overline f_r]_{\Sigma}} = \gamma_{{\rm ad}^{\Sigma}_{f_0^S}(f_0+ \overline f_1+\ldots + \overline f_r)} =0\, .
\end{align}
Hence, $\left. \frac{d}{dt}\right|_{t=0}(e^{t\gamma_{f_0^S}})_*( \gamma_{f_0+ \overline f_1 \ldots +\overline f_r})  = 0$ and thus the truncated normal form  
commutes with the flow $t\mapsto e^{t\gamma_{f_0^S}}$ of the coupled cell network vector field $\gamma_{f_0^S}$. 
\end{proof}
The continuous family 
$$t\mapsto e^{t\gamma_{f_0^S}}$$ 
of transformations of $V^N$ is called a {\it normal form symmetry}. This symmetry is sometimes used to characterize vector fields that are in normal form. It also plays an important role in finding periodic solutions near equilibria of the vector field $\gamma_f$, using for example the method of Lyapunov-Schmidt reduction \cite{duisbif},  \cite{constrainedLS}, \cite{golschaef1}, \cite{perspective}, \cite{golschaef2}.
 
\section{A fundamental semigroup network}\label{coveringsection}
As a byproduct of Theorem \ref{Amaps}, and perhaps as a curiosity, we will show in this section that the dynamics of $\gamma_f$ on $V^N$ is conjugate to the dynamics of a certain network $\Gamma_f$ on $V^n$. We will argue that $\Gamma_f$ acts as a ``fundamental network'' for $\gamma_f$.

We recall that if $\Sigma=\{\sigma_1, \ldots, \sigma_n\}$ is a semigroup, then every $\sigma_j\in \Sigma$ induces a map 
$$\widetilde \sigma_j: \{1,\ldots, n\}\to\{1, \ldots, n\}\ \mbox{via the formula}\ \sigma_{\widetilde \sigma_j(k)} = \sigma_k \circ \sigma_j\, .$$
We saw that $\widetilde{\sigma_{j_1}\circ\sigma_{j_2}} =\widetilde \sigma_{j_1}\circ \widetilde \sigma_{j_2}$ and hence the collection $\widetilde \Sigma:=\{\widetilde \sigma_1, \ldots, \widetilde \sigma_n\}$ is closed under composition. One can now study coupled cell networks subject to $\widetilde \Sigma$. They have the form
$$\Gamma_f:V^n\to V^n\ \mbox{with} \ (\Gamma_{f})_j(X) := f(X_{\widetilde \sigma_1(j)}, \ldots, X_{\widetilde \sigma_n(j)}) = f(A_{\sigma_j}X) \ \mbox{for all} \ 1\leq j\leq n\, .$$
The following theorem demonstrates that $\gamma_f$ and $\Gamma_f$ are dynamically related:
\begin{theorem}\label{conjugation}
All maps $\pi_i:V^N\to V^n$ conjugate $\gamma_f$ to $\Gamma_f$, that is 
$$\Gamma_f\circ \pi_i = \pi_i\circ \gamma_f\  \ \mbox{for all} \ 1\leq i\leq N \, .$$
\end{theorem}
\begin{proof}
For $x\in V^N$ we have that 
$$(\Gamma_f\circ\pi_i)_j(x)=f((A_{\sigma_j}\circ \pi_i)(x)) =f(\pi_{\sigma_j(i)}(x)) = (\gamma_f(x))_{\sigma_j(i)} = (\pi_i\circ\gamma_f)_j(x)\, .$$
\end{proof}
Theorem \ref{conjugation} implies that every $\pi_i$ sends integral curves of $\gamma_f$ to integral curves of $\Gamma_f$ and discrete-time orbits of $\gamma_f$ to discrete-time orbits of $\Gamma_f$.

In addition, the dynamics of $\gamma_f$ can be reconstructed from the dynamics of $\Gamma_f$. More precisely, when $X_{(i)}(t)$ are integral curves of $\Gamma_f$ with $X_{(i)}(0)=\pi_i(x(0))$, then an integral curve $x(t)$ of $\gamma_f$ can simply be obtained by integration of the equations
$$\dot x_i(t) = f(X_{(i)}(t)) \ \mbox{for}\ 1\leq i\leq N\, .$$ 
Similarly, if $X_{(i)}^{(m)}$ are discrete-time orbits of $\Gamma_f$ with $X_{(i)}^{(m)}(0)=\pi_i(x(0))$, then $x_i^{(m+1)}:=f(X_{(i)}^{(m)})$ defines a discrete-time orbit of $\gamma_f$.

The transition from $\gamma_f$ to $\Gamma_f$ is thus reminiscent of the symmetry reduction of an equivariant dynamical system: the dynamics of $\gamma_f$ descends to the dynamics of $\Gamma_f$ and the dynamics of $\gamma_f$ can be reconstructed from that of $\Gamma_f$ by means of integration. 
Nevertheless,
$n$ can of course be both smaller and larger than $N$. In the latter case, the dynamics of $\Gamma_f$ may be much richer than that of $\gamma_f$ and it is confusing to speak of reduction. In either case, $\Gamma_f$ captures all the dynamics of $\gamma_f$.

\begin{example}
Again, let $N=3$ and let $\sigma_1, \sigma_2, \sigma_3$ be defined as in Example \ref{ex2}. Recall
\begin{align}\nonumber
& A_{\sigma_1}(X_1, X_2, X_3) = (X_1, X_2, X_3) \, , \\ \nonumber & A_{\sigma_2}(X_1, X_2, X_3) = (X_2, X_3, X_3) \, , \\ 
& A_{\sigma_3}(X_1, X_2, X_3) = (X_3, X_3, X_3) \, . \nonumber
\end{align}
This means that the network map $\Gamma_f$ is given by
$$\Gamma_f(X_1, X_2, X_3)=(f(X_1, X_2, X_3), f(X_2, X_3, X_3), f(X_3, X_3, X_3) )\, .$$
In this example, the conjugacies from $\gamma_f$ to $\Gamma_f$ are
\begin{align}
&\pi_1:(x_1, x_2, x_3)\mapsto (X_1, X_2, X_3):=(x_1, x_1, x_1)\, ,\nonumber \\
&\pi_2:(x_1, x_2, x_3)\mapsto (X_1, X_2, X_3):=(x_2, x_1, x_1)\, ,\nonumber \\
&\pi_3:(x_1, x_2, x_3)\mapsto (X_1, X_2, X_3):=(x_3, x_2, x_1)\, . \nonumber
\end{align}
The conjugacy $\pi_3$ is bijective, which explains that Figures \ref{pict2} and \ref{pict5} are isomorphic.
 \begin{figure}[ht]\renewcommand{\figurename}{\rm \bf \footnotesize Figure}
\centering

\begin{tikzpicture}[->,>=stealth',shorten >=1pt,auto,node distance=2cm,
                    thick,main node/.style={circle,draw,font=\sffamily\Large\bfseries}]

  \node[main node] (1) {$X_3$};
  \node[main node] (2) [below of=1] {$X_2$};
  \node[main node] (3) [below of=2] {$X_1$};

   \node[main node] (4) [right of=1] {$X_3$};
  \node[main node] (5) [below of=4] {$X_2$};
  \node[main node] (6) [below of=5] {$X_1$};
  \node[main node] (7) [right of =4] {$X_3$};
  \node[main node] (8) [below  of=7] {$X_2$};
  \node[main node] (9) [below of=8] {$X_1$};

  \path[every node/.style={font=\sffamily\small}]
   (2) edge [loop above] node {} (2)
 (3) edge [loop above] node {} (3)
(1) edge [loop above] node {} (1)
    
    (5) edge node {} (6)
 (4) edge  node {} (5)
(4) edge [loop above] node {} (4)
 (7) edge [loop above] node {} (7)
 (7) edge node {} (8)
(7) edge [bend left] node {} (9)
    ;

\end{tikzpicture}

\caption{\footnotesize {\rm The fundamental network of our three-cell feedforward network.}}
\label{pict5}
\end{figure}
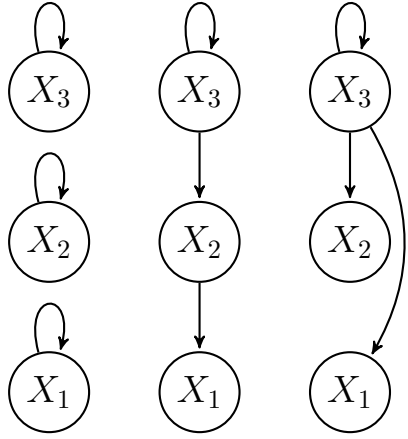

\end{example}

\begin{example}
Recall Example \ref{funnyexample} that features a semigroup with $4$ elements. 
To determine the maps $A_{\sigma_j}$, we compute the multiplication table of this semigroup:     
$$\begin{array}{c|cccc} 
\circ & \sigma_1 & \sigma_2 & \sigma_3 & \sigma_4\\
\hline
\sigma_1 & \sigma_1 &\sigma_1 & \sigma_1&\sigma_1\\
\sigma_2 & \sigma_2 &\sigma_2 & \sigma_2&\sigma_2\\
\sigma_3 & \sigma_2 &\sigma_1 & \sigma_4&\sigma_3\\
\sigma_4 & \sigma_1 &\sigma_2 & \sigma_3&\sigma_4
\end{array}
\begin{array}{l} \\ \\ \\ \\ . \end{array}
$$
This implies that
\begin{align}\nonumber
A_{\sigma_1}(X_1, X_2, X_3, X_4)=(X_1, X_2, X_2, X_1), \\ \nonumber
 A_{\sigma_2}(X_1, X_2, X_3, X_4)=(X_1, X_2, X_1, X_2), \\ \nonumber
A_{\sigma_3}(X_1, X_2, X_3, X_4)=(X_1, X_2, X_4, X_3), \\ \nonumber
A_{\sigma_4}(X_1, X_2, X_3, X_4)=(X_1, X_2, X_3, X_4). 
\end{align}
Hence, the corresponding fundamental network is given by
 \begin{align}\nonumber
\begin{array}{ll} 
\dot X_1 =& f(X_1, X_2, X_2, X_1)\, , \\
\dot X_2 =& f(X_1, X_2, X_1, X_2)\, , \\
\dot X_3= & f(X_1, X_2, X_4, X_3)\, , \\
\dot X_4 = & f(X_1, X_2, X_3, X_4)\, .
\end{array}
\end{align}
This fundamental network has been depicted in Figure \ref{pict9}.
 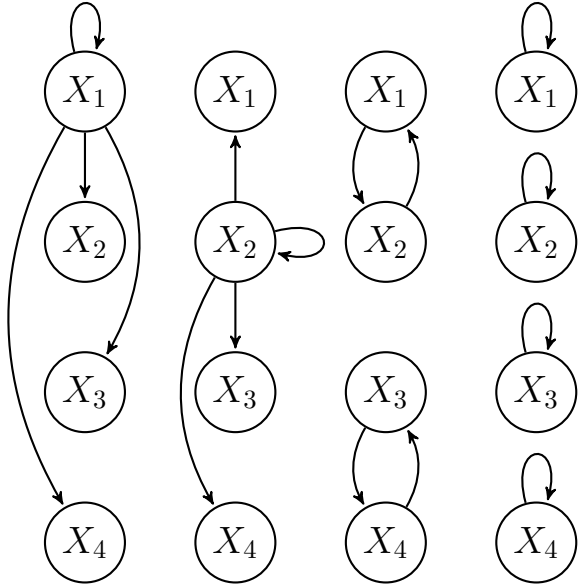
\begin{figure}[ht]\renewcommand{\figurename}{\rm \bf \footnotesize Figure} 
\centering
\begin{tikzpicture}[->,>=stealth',shorten >=1pt,auto,node distance=2cm,
                    thick,main node/.style={circle,draw,font=\sffamily\Large\bfseries}]

  \node[main node] (1) {$X_1$};
  \node[main node] (2) [below of=1] {$X_2$};
  \node[main node] (3) [below of=2] {$X_3$};
  \node[main node] (4) [below of=3] {$X_4$};

   \node[main node] (5) [right of=1] {$X_1$};
  \node[main node] (6) [below of=5] {$X_2$};
  \node[main node] (7) [below of=6] {$X_3$};
  \node[main node] (8) [below of=7] {$X_4$};

  \node[main node] (9) [right of =5] {$X_1$};
  \node[main node] (10) [below  of=9] {$X_2$};
  \node[main node] (11) [below of=10] {$X_3$};
   \node[main node] (12) [below of=11] {$X_4$};

\node[main node] (13) [right of =9] {$X_1$};
  \node[main node] (14) [below  of=13] {$X_2$};
  \node[main node] (15) [below of=14] {$X_3$};
 \node[main node] (16) [below of=15] {$X_4$};

  \path[every node/.style={font=\sffamily\small}]
   (1) edge [loop above] node {} (1)
 (1) edge node {} (2)
(1) edge [bend left] node {} (3)
 (1) edge [bend right] node {} (4)   
    (6) edge node {} (5)
 (6) edge  [loop right] node {} (6)
(6) edge  node {} (7)
 (6) edge [bend right ] node {} (8)
 (9) edge [bend right] node {} (10)
(10) edge [bend right] node {} (9)
(11) edge [bend right] node {} (12)
  (12) edge [bend right] node {} (11)  
 (13) edge [loop above] node {} (13)
 (14) edge [loop above] node {} (14)
 (15) edge [loop above] node {} (15)
 (16) edge [loop above] node {} (16)
;
\end{tikzpicture}
\caption{\footnotesize {\rm A fundamental network with $n=4$.}}
\label{pict9}
\end{figure}
It is clear that the maps 
\begin{align}
&\pi_1:(x_1, x_2)\mapsto (X_1, X_2, X_3, X_4):=(x_1, x_2, x_2, x_1)\, ,\nonumber \\
&\pi_2:(x_1, x_2)\mapsto (X_1, X_2, X_3, X_4):=(x_1, x_2, x_1, x_2)\, \nonumber 
\end{align}
conjugate $\gamma_f$ to $\Gamma_f$.
\end{example}
The advantage of studying $\Gamma_f$ instead of $\gamma_f$ is that 
the definition $(\Gamma_f)_j:=f\circ A_{\sigma_j}$ explicitly displays the representation of the semigroup $\Sigma$, whereas the definition $(\gamma_f)_i:=f\circ\pi_i$ clearly does not. This has as a consequence that the transformation formulas for the composition and the Lie bracket become completely natural: 
\begin{lemma}\label{gammaGammacomposition}
It holds that 
$$\Gamma_f\circ \Gamma_g= \Gamma_{f\circ_{\Sigma}g}\ \mbox{and}\ [\Gamma_f, \Gamma_g]=\Gamma_{[f,g]_{\Sigma}}\, .$$
\end{lemma}
\begin{proof}
First of all,
\begin{align}\nonumber
A_{\sigma_j}(g(A_{\sigma_1}X), &\ldots, g(A_{\sigma_n}X))= (g(A_{\sigma_{\widetilde \sigma_1(j)}}X), \ldots, g(A_{\sigma_{\widetilde \sigma_n(j)}}X))\\ \nonumber &= (g(A_{\sigma_1\circ \sigma_j}X), \ldots, g(A_{\sigma_n\circ\sigma_j}X))\, . 
\end{align}
This gives that
\begin{align}\nonumber
(\Gamma_f\circ\Gamma_g)_j(X)=f(A_{\sigma_j}(g(A_{\sigma_1}X), \ldots, g(A_{\sigma_n}X)) &= f(g(A_{\sigma_1\circ \sigma_j}X), \ldots, g(A_{\sigma_n\circ \sigma_j}X)) = \\
\nonumber
  f(g(A_{\sigma_1}(A_{\sigma_j}X)), \ldots, g(A_{\sigma_n}(A_{\sigma_j}X))) =& (f\circ_{\Sigma}g)(A_{\sigma_j}X) = (\Gamma_{f\circ_{\Sigma}g})_j(X)\, .
\end{align}
The computation for the Lie bracket is similar.
\end{proof} 
We stress that Lemma \ref{gammaGammacomposition} holds due to the definition $(\Gamma_f)_j:=f\circ A_{\sigma_j}$ and the fact that $\sigma_j\mapsto A_{\sigma_j}$ is a homomorphism. Lemma \ref{gammaGammacomposition} implies for example that the symbolic computation of the normal form of $\Gamma_f$ is the same as the symbolic computation of the normal form of $\gamma_f$.

We propose to call $\Gamma_f$  the {\it fundamental network} of $\gamma_f$. Two properties make this fundamental network fundamental: first of all, the network architecture of the fundamental network only depends on the multiplicative structure of the semigroup $\Sigma$ and not on the explicit realization of $\Sigma$ itself - in particular, it does not depend on $N$. This means that two semigroup networks have isomorphic fundamental networks if and only if their semigroups are isomorphic. The second fundamental property of the fundamental network is that it is equal to its own fundamental network, if the latter is defined. This follows from Proposition \ref{faithful} below, in which we call a homomorphism of semigroups {\it faithful} if it is injective.
\begin{proposition}\label{faithful}
Assume that the homomorphism $\sigma_j\mapsto \widetilde \sigma_j$ is faithful. Then
 $$\widetilde{\widetilde \sigma}_j = \widetilde \sigma_j\ \mbox{and therefore}\ A_{\sigma_j}=A_{\widetilde \sigma_j}\ \mbox{for all} \ 1\leq j\leq n\, .$$
\end{proposition}
\begin{proof}
Recall that $\widetilde \Sigma=\{\widetilde \sigma_1, \ldots, \widetilde \sigma_n\}$ is closed under composition. Thus, the condition that the homomorphism $\sigma_j\mapsto \widetilde \sigma_j$ from $\Sigma$ to $\widetilde\Sigma$ is faithful just means that $\widetilde \Sigma$ is a semigroup. In particular, each map $\widetilde{\widetilde \sigma}_j:\{1, \ldots, n\}\to\{1,\ldots, n\}$ is then well-defined. Now we compute
$$\widetilde \sigma_{\widetilde{\widetilde \sigma}_j(k)} = \widetilde \sigma_j\circ \widetilde \sigma_k = \widetilde{\sigma_j\circ \sigma_k}=\widetilde{\sigma_{\widetilde \sigma_j(k)}}=\widetilde \sigma_{\widetilde \sigma_j(k)}\, .$$
This proves that $\widetilde{\widetilde \sigma}_j = \widetilde \sigma_j$ for all $1\leq j\leq n$ and hence that $A_{\widetilde \sigma_j}=A_{\sigma_j}$ for all $1\leq j\leq n$.
\end{proof}
 Proposition \ref{faithful} brings up the question when the homomorphism $\sigma_j\mapsto \widetilde \sigma_j$ is faithful, i.e. under which conditions the elements of $\Sigma$ all have different left-multiplicative behavior. We give a partial answer to this question in Remark \ref{slavesremark} below. The upshot of this remark is that one may essentially always assume the homomorphism to be faithful.

We finish this section with a few simple observations on synchrony and symmetry for $\Gamma_f$. First of all, a direct consequence of Theorem \ref{conjugation} is that each ${\rm im}\, \pi_i\subset V^n$ is an invariant subspace for the dynamics of $\Gamma_f$. Interestingly, another way to see this is by the following
\begin{proposition}
Every ${\rm im}\, \pi_i \subset V^n$ is a robust synchrony space for the $\Gamma_f$'s.
\end{proposition}
\begin{proof}
Let us define a partition $P$ of $\{1, \ldots, n\}$ by letting $1\leq j_1, j_2\leq n$ be in the same element of $P$ if and only if $\sigma_{j_1}(i)=\sigma_{j_2}(i)$. Then
$${\rm Syn}_P\! =\! \{X\in V^n\, | \, X_{j_1}=X_{j_2} \, \mbox{when}\, \sigma_{j_1}(i)=\sigma_{j_2}(i)\}= \{ (x_{\sigma_1(i)}, \ldots, x_{\sigma_n(i)}) \, | \, x\in V^N\} = {\rm im}\, \pi_i\ .$$
It remains to show that the partition $P$ is balanced for $\widetilde \Sigma$. This is easy though: when $1\leq j_1, j_2\leq n$ are in the same element of $P$, then it holds for all $1\leq k\leq n$ that 
$$\sigma_{\widetilde \sigma_{k}(j_1)} (i)= (\sigma_k\circ \sigma_{j_1})(i)=(\sigma_{k}\circ \sigma_{j_2})(i) = \sigma_{\widetilde \sigma_{k}(j_2)}(i)\, ,$$
where the middle equality holds because $\sigma_{j_1}(i)=\sigma_{j_2}(i)$.
This proves that also $\widetilde \sigma_{k}(j_1)$ and $\widetilde \sigma_{k}(j_2)$ are in the same element of $P$ and hence that the elements of $\widetilde \Sigma$ preserve $P$.
\end{proof}
Recall that $A_{\sigma_j}\circ \pi_i=\pi_{\sigma_j(i)}$. This implies that $A_{\sigma_j}$ sends the $\Gamma_f$-invariant subspace ${\rm im}\, \pi_i$ to the $\Gamma_f$-invariant subspace ${\rm im}\, \pi_{\sigma_j(i)}$. But much more is true: the following result shows that $A_{\sigma_j}$
sends all orbits of $\Gamma_f$ to orbits of $\Gamma_f$, even though $A_{\sigma_j}$ may not be invertible.
 \begin{proposition}
$$\Gamma_f\circ A_{\sigma_j}=A_{\sigma_j}\circ \Gamma_f\ .$$
\end{proposition}
\begin{proof}
\begin{align}
(\Gamma_f\circ A_{\sigma_j})_k(X) &= f(A_{\sigma_k}\circ A_{\sigma_j}X) = f(A_{\sigma_k\circ\sigma_j}X)= \nonumber \\ \nonumber f(A_{\sigma_{\widetilde \sigma_k(j)}}X) &= (\gamma_f)_{\widetilde \sigma_k(j)}(X) = (A_{\sigma_j}\circ \gamma_f)_k(X)\ .
\end{align}
\end{proof}
The final result of this section shows that $\Gamma_f$ may even have more symmetry: the dynamical input symmetries of $\gamma_f$ are true symmetries of $\Gamma_f$.
\begin{proposition}
If $p$ is a permutation of $\{1, \ldots, N\}$ and $q$ is a permutation of $\{1, \ldots, n\}$ so that $p\circ \sigma_j=\sigma_{q(j)}\circ p$ for all $1\leq j\leq n$ and if $f\circ \lambda_q \circ \pi_i=f\circ \pi_i$ for all $1\leq i \leq N$, then 
$$\Gamma_f\circ \lambda_q  = \lambda_q\circ\Gamma_f\ \mbox{on every}\ {\rm im}\, \pi_i\, .$$ 
\end{proposition}
\begin{proof}
Recall that under the conditions of the proposition, it holds that $\pi_i\circ \lambda_p=\lambda_q\circ \pi_{p(i)}$ and that from this it followed that $\lambda_{p}\circ \gamma_f=\gamma_f\circ \lambda_p$. As a consequence,
$$\Gamma_f\circ \lambda_q\circ \pi_{p(i)}=\Gamma_f\circ \pi_i\circ \lambda_p = \pi_i \circ \gamma_f \circ \lambda_p = \pi_i\circ \lambda_p\circ \gamma_f=\lambda_q\circ \pi_{p(i)}\circ \gamma_f=\lambda_q\circ \Gamma_f\circ\pi_{p(i)}\, .$$
Because $p$ is a permutation, this means that $\gamma_f\circ \lambda_q  = \lambda_q\circ\gamma_f$ on every ${\rm im}\, \pi_i$.
\end{proof}

\begin{remark}\label{slavesremark}
To explain when the homomorphism $\sigma_j\mapsto \widetilde \sigma_j$ is faithful, we can make the following definition:
we say that $1\leq i \leq N$ is a {\it slave} for the network $\Sigma$ if there are no $1\leq j \leq n$ and $1\leq k\leq N$ so that 
$\sigma_j(k)=i$. Thus, a slave is a cell that does not act as input for any other cell, not even for itself. The point of this definition is the following:
\begin{proposition}
If $\Sigma$ has no slaves, then $\sigma_j\mapsto \widetilde \sigma_j$ is a faithful homomorphism.
\end{proposition}
\begin{proof} The relation $\widetilde \sigma_{j_1}=\widetilde \sigma_{j_2}$ means that $\sigma_{j_1}\circ\sigma_k=\sigma_{j_2}\circ\sigma_k$ for all $k$. This implies in particular that $\sigma_{j_1}=\sigma_{j_2}$ on ${\rm im}\, \sigma_{k}$ for all $k$. But
if $\Sigma$ is free of slaves, then $\bigcup_{j=1}^n {\rm im}\, \sigma_j=\{1,\ldots, N\}$. Hence, $\sigma_{j_1}=\sigma_{j_2}$.
\end{proof}
If a network has slaves, then we can reduce it until no slaves remain. This works as follows: first of all, we remove any slave from the network. Because slaves do not affect the dynamics of other cells, this can be done without any effect on the network dynamics.
Removing slaves may create new slaves: these are the cells that acted as inputs only for the original slaves. These new slaves can also be removed, etc. until a network free of slaves remains. 

The remaining network may not be defined unambiguously, because some of the maps in $\Sigma$ may coincide after the removal of the slaves. This happens when distinct maps in $\Sigma$ differ only at slaves. Such maps can be identified though, while $f$ must be redefined. In this way, we produce an unambiguous network that is free of slaves. For such a network $\gamma_f$, the corresponding $\Gamma_f$ is a true fundamental network.
\end{remark}
 
\section{Some examples and their normal forms}\label{examplessection}
In this section we illustrate the methods and results of this paper by computing the normal forms of two coupled cell networks. Keeping things simple, we restrict our attention to synchrony breaking steady state bifurcations in one-parameter families of networks with one-dimensional cells.
\subsection{A skew product network}
In the first example, we consider the homogeneous skew product differential equations
\begin{align}\label{fullexample1}
\begin{array}{ll} 
\dot x_1 = &  f(x_1, x_1; \lambda)\, , \\
\dot x_2 = & f(x_2,x_1; \lambda)\, .
\end{array}\, 
\end{align}
Here $x_1, x_2\in \R$ and $f:\R^2\times \R\to \R$. As usual, we will denote the right hand side of (\ref{fullexample1}) by $\gamma_f(x_1,x_2;\lambda)$ and we will henceforth assume that 
$$\gamma_f(0,0;0)=0\ \mbox{and}\ D_x\gamma_f(0,0;0)\ \mbox{is not invertible}.$$ 
This means that at the parameter value $\lambda=0$, the origin $(x_1, x_2)=(0,0)$ is a fully synchronous equilibrium point of (\ref{fullexample1}) that undergoes a steady state bifurcation. We wish to study the generic nature of this bifurcation. So let us write 
$$f_{0,0}(X_1, X_2)=D_Xf(0,0;0)(X_1, X_2)=a_1X_1+a_2X_2\ \mbox{with}\ a_1, a_2\in \R\, .$$ 
With this notation, we have that
$${\rm mat}\, D_x\gamma_{f}(0,0;0) = \left( \begin{array}{cc} a_1 +a_2 & 0\\ a_2 & a_1 \end{array}\right)  \, .$$
We remark that this linearization matrix is semisimple. And moreover that a steady state bifurcation occurs when one of its eigenvalues $a_1+a_2$ or $a_1$ vanishes. 

The obvious but important remark is now that equations (\ref{fullexample1}) define a semigroup coupled cell network. The corresponding semigroup consists of $\sigma_1$ and $\sigma_2$, where
$$\sigma_1(1)= 1, \sigma_1(2) = 2\ \mbox{and} \ \sigma_2(1)= 1, \sigma_2(2) = 1\, .$$
We depicted this network in Figure \ref{pict6}.
 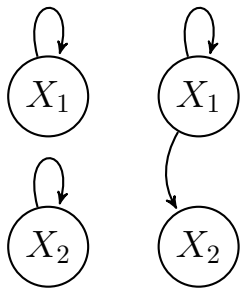
\begin{figure}[ht]\renewcommand{\figurename}{\rm \bf \footnotesize Figure} 
\centering
\begin{tikzpicture}[->,>=stealth',shorten >=1pt,auto,node distance=2cm,
                    thick,main node/.style={circle,draw,font=\sffamily\Large\bfseries}]

  \node[main node] (1) {$X_1$};
  \node[main node] (2) [below of=1] {$X_2$};
  \node[main node] (3) [right of=1] {$X_1$};
  \node[main node] (4) [below of=3] {$X_2$};

  \path[every node/.style={font=\sffamily\small}]
   (1) edge [loop above] node {} (1)
 (2) edge [loop above] node {} (2)
(3) edge [loop above] node {} (3)
 (3) edge [bend right] node {} (4)   
   
;
\end{tikzpicture}
\caption{\footnotesize {\rm A homogeneous skew product network.}}
\label{pict6}
\end{figure}

\noindent The composition table of $\{\sigma_1, \sigma_2\}$ reads
$$\begin{array}{c|cc} 
\circ & \sigma_1 & \sigma_2  \\
\hline
\sigma_1 & \sigma_1 &\sigma_2  \\
\sigma_2 & \sigma_2 &\sigma_2   
\end{array}\, 
\begin{array}{l} \\ \\ . \end{array}
$$
From this table, we can read off that 
$$A_{\sigma_1}(X_1, X_2) = (X_1, X_2) \ , \ A_{\sigma_2}(X_1, X_2) = (X_2, X_2)\, ,$$
and hence that the symbolic Lie bracket of this network is given by
\begin{align}\nonumber
[f, g]_{\Sigma}(X_1, X_2) = & D_1f(X_1,X_2)g(X_1, X_2) + D_2f(X_1,X_2)g(X_2, X_2) \\ \nonumber
 - &  D_1g(X_1,X_2)f(X_1, X_2) + D_2g(X_1,X_2)f(X_2, X_2) \, .\end{align}
Since $D_{x}\gamma_f(0,0;0)$ is semisimple, so is ${\rm ad}^{\Sigma}_{f_{0,0}}: P^{k,l} \to P^{k,l}$ for every $k\geq -1$ and $l\geq 0$. Its kernel determines the normal form of $f$. It only requires a little computation to check that
\begin{align}\nonumber
{\rm ad}^{\Sigma}_{f_{0,0}}:\left\{\!\! \begin{array}{lll} (X_1-X_2)^{\alpha}X_2^{\beta}  &\!\! \mapsto [(1-\alpha)a_1-\beta(a_1+a_2)] (X_1-X_2)^{\alpha}X_2^{\beta} & \!\!\mbox{for}\ \alpha\geq 1\ \mbox{and}\   \beta\geq 0\, ,\\ 
 X_2^{\beta} &\!\! \mapsto (1-\beta)(a_1+a_2)X_2^{\beta} &\!\! \mbox{for}\ \beta\geq 0\, .
 \end{array}\right.
\end{align}
This formula nicely confirms that ${\rm ad}^{\Sigma}_{f_{0,0}}$ is semisimple.
We now consider the two different codimension one cases:
\begin{itemize} 
\item[{\bf 1.}] When $a_1+a_2=0$ but $a_1\neq 0$ then the kernel of $D_x\gamma_f(0,0;0)$ is tangent to the synchrony space $\{x_1=x_2\}$. In this case, the kernel of ${\rm ad}^{\Sigma}_{f_{0,0}}$ is spanned by elements of the form $(X_1-X_2)X_2^{\beta}$ and $X_2^{\beta}$, where $\beta\geq 0$. 
Thus, the general normal form of $f$ is 
$$\overline f(X_1, X_2; \lambda)= (X_1-X_2)F(X_2;\lambda) + G(X_2; \lambda) \, ,$$ 
with $F(X_2;\lambda) = A(\lambda) + \mathcal{O}(X_1)$, $G(X_1; \lambda)= B(\lambda) +C(\lambda)X_1+ D(\lambda)X_1^2 + \mathcal{O}(X_1^3)$ and $A(0)=a_1, B(0) =C(0) = 0$. The normal form equations of motion become
\begin{align}\nonumber
\begin{array}{ll} 
\dot x_1 = & G(x_1; \lambda)\, ,  \\
\dot x_2 = & G(x_1; \lambda)+ (x_2-x_1)F(x_1;\lambda) \, .
\end{array}
\end{align}
This implies first of all that the stationary points of the normal form satisfy $x_1=x_2$ and secondly that $x_1$ solves the equation $G(x_1;\lambda)=B(\lambda)+C(\lambda)x_1+D(\lambda)x_1^2 + \mathcal{O}(x_1^3)=0$.
Under the generic conditions that $B'(0), D(0)\neq 0$, we thus find the saddle node branches
$$ x_1=x_2 =\pm\sqrt{(-B'(0)/D(0))\lambda} + \mathcal{O}(\lambda)$$ of synchronous steady states. A straightforward stability analysis reveals that one of these branches consists of equilibria that are linearly stable in the direction of the synchrony space, while the other branch consists of unstable points. We remark that the saddle node bifurcation is also generic in codimension one in the context of vector fields without any special structure. 

\item[{\bf 2.}]  When $a_1+a_2\neq 0$ and $a_1=0$, then the kernel of ${\rm ad}^{\Sigma}_{f_{0,0}}$ is spanned by elements of the form $(X_1-X_2)^{\alpha}$, where $\alpha \geq 1$, and the element $X_2$. Hence the general normal form of $f$ is given by 
$$\overline f(X_1, X_2; \lambda)= (X_1-X_2)F(X_1-X_2; \lambda) + A(\lambda)X_2\, ,$$ 
with $F(X_1-X_2; \lambda) = B(\lambda) +C(\lambda)(X_1-X_2)+\mathcal{O}(X_1-X_2)^2$ and $A(0)= a_2, B(0)=0$. The normal form differential equations are
\begin{align}\nonumber
\begin{array}{ll} 
\dot x_1 = &   A(\lambda)x_1\, , \\
\dot x_2 = & A(\lambda)x_1 + (x_2-x_1)F(x_2-x_1; \lambda) \, .
\end{array}
\end{align}
This implies that the stationary points of the normal form satisfy $x_1=0$, while either $x_2=0$ or $x_2$ solves the equation $F(x_2;\lambda)=B(\lambda)+C(\lambda)x_2+\mathcal{O}(x_2^2)=0$.
Under the generic conditions that $B'(0), C(0)\neq 0$, we thus find the two steady state branches 
$$x_1 = x_2=0\ \mbox{and}\ x_1=0, x_2=(-B'(0)/C(0))\lambda + \mathcal{O}(\lambda^2)\, .$$ 
\noindent These branches exchange stability when they cross. This means that the normal form displays a synchrony breaking transcritical bifurcation. Such a bifurcation is not generic in codimension one in the context of vector fields without any special structure, and is hence forced by the network structure. More precisely, it follows from the presence of the invariant synchrony space.
 \end{itemize}

  \subsection{A nilpotent feed-forward network}
Next, we consider differential equations with the network structure defined in Example \ref{ex2}: 
\begin{align}\label{example}
 \begin{array}{c} \dot x_1 = f(x_1,x_1,x_1; \lambda) \, ,\\ \dot x_2 = f(x_2,x_1,x_1; \lambda)\, , \\ \dot x_3 = f(x_3, x_2, x_1; \lambda) \, .\end{array}
\end{align}
Here $x_1, x_2, x_3\in\R$ and $f:\R^3\times\R\to\R$. Again, let us write
$$f_{0,0}(X_1, X_2, X_3) = D_Xf(0,0,0;0)(X_1, X_2, X_3)=a_1X_1+a_2X_2+a_3X_3\ \mbox{for} \ a_1, a_2, a_3\in\R\, .$$
Then it holds that 
$${\rm mat}\, D_x\gamma_{f}(0,0,0;0)=  \left( \begin{array}{rrr} a_1+ a_2+ a_3 & 0 & 0 \\ a_2+a_3 & a_1 & 0 \\ a_3 & a_2 & a_1 \end{array}\right)\, .$$
This shows that a steady state bifurcation takes place when either $a_1+a_2+a_3=0$ or $a_1=0$. Moreover, the linearization matrix is not semisimple. In fact, its SN-decomposition reads
$$\left( \begin{array}{rrr} a_1 + a_2 +a_3 & 0 & 0 \\ a_2+a_3 & a_1 & 0 \\ a_3 & a_2& a_1 \end{array}\right)  = \left( \begin{array}{rrr} a_1 + a_2+a_3 & 0 & 0 \\ a_2+a_3 & a_1 & 0 \\ a_2+a_3& 0 & a_1 \end{array}\right) + \left( \begin{array}{rrr} 0 & 0 & 0 \\ 0 & 0 & 0 \\ -a_2 & a_2 & 0 \end{array}\right) \, .  $$
As a consequence, we should accordingly decompose $f_{0,0}$ as 
$$f_{0,0}=f_{0,0}^S+f_{0,0}^N\ \mbox{where}\ f_{0,0}^S(X_1,X_2,X_3) \!= \! a_1X_1+(a_2+a_3)X_3, f_{0,0}^N(X_1,X_2,X_3)\!=\!a_2(X_2-X_3)\, .$$
Recalling that for this network the expression for the symbolic bracket is given in Example \ref{exzoveel}, it again requires a little computation to find that
\begin{align}\nonumber
\mbox{ad}_{f_{0,0}^S}\!: \!\! \left\{ \!\!\!\! \begin{array}{lll} X_3^{\gamma} &\mapsto (1-\gamma)(a_1+a_2+a_3)X_3^{\gamma}  & \hspace{-.1cm} \mbox{for} \ \gamma\geq 0\, ,\\ (X_1-X_3)^{\alpha}(X_2-X_3)^{\beta}X_3^{\gamma}  &\mapsto &  \\   & \hspace{-3.5cm} [(1\!-\!\alpha\!-\!\beta)a_1\! -\! \gamma(a_1\!+\!a_2\!+\!a_3)](X_1-X_3)^{\alpha}(X_2-X_3)^{\beta}X_3^{\gamma} & \hspace{-.1cm} \mbox{for} \ \alpha+\beta\geq 1, \gamma\geq 0\, .
  \end{array}\right.
 \end{align}
and similarly that
\begin{align}\nonumber
 \mbox{ad}_{f_{0,0}^N}: \left\{ \begin{array}{lll}\!\!\! (X_1-X_3)^{\alpha}(X_2-X_3)^{\beta}X_3^{\gamma} & \mapsto & \\ & \hspace{-3.5cm} - \alpha a_2 (X_1-X_3)^{\alpha-1}(X_2-X_3)^{\beta+1}X_3^{\gamma}  & \hspace{-.1cm} \mbox{for} \ \alpha, \beta \geq 1, \gamma\geq 0\, , \\
\!\!\! (X_1-X_3)^{\alpha}X_3^{\gamma}  &\mapsto & \\ & \hspace{-3.5cm} a_2(X_2-X_3)^{\alpha}X_3^{\gamma} - \alpha a_2(X_1-X_3)^{\alpha-1}(X_2-X_3)X_3^{\gamma} &\hspace{-.1cm} \mbox{for}\ \alpha \geq 1, \gamma \geq 0\, ,\\
\!\!\!  (X_2-X_3)^{\beta}X_3^\gamma & \mapsto 0 &\hspace{-.1cm} \mbox{for} \ \beta, \gamma\geq 0\, .
  \end{array}
 \right.
\end{align}
Once more, we now consider the two codimension one cases:
\begin{itemize}
\item[{\bf 1.}] 
If $a_1+a_2+a_3= 0$ and $a_1\neq 0$, then the kernel of ${\rm ad}^{\Sigma}_{f_{0,0}}$ is spanned by terms
$$(X_1-X_3)X_3^{\gamma}, (X_2-X_3)X_3^{\gamma}\ \mbox{and}\ X_3^{\gamma}\ \mbox{with}\ \gamma\geq 0\, .$$
One checks that ${\rm ad}_{f_{0,0}^N}$ vanishes on this kernel, so the general normal form of $f$ is
$$\overline f(X_1, X_2, X_3; \lambda) = (X_1-X_3)F(X_3;\lambda) + (X_2-X_3) G(X_3;\lambda) + H(X_3;\lambda)\, ,$$
where $F(X_3;\lambda)=A(\lambda) + \mathcal{O}(X_3)$, $G(X_3;\lambda)=B(\lambda) + \mathcal{O}(X_3)$, $H(X_3;\lambda)=C(\lambda) + D(\lambda)X_3 +E(\lambda)X_3^2 + \mathcal{O}(X_3^3)$ and $A(0)=a_1, B(0)=a_2, C(0)=D(0)=0$.
The normal form equations of motion are
\begin{align}
 \begin{array}{l} \dot x_1 = H(x_1; \lambda) \, ,\\ \dot x_2 =  (x_2-x_1)F(x_1;\lambda) +H(x_1;\lambda)\, , \\ \dot x_3 =  (x_3-x_1)F(x_1;\lambda) + (x_2-x_1)G(x_1;\lambda) + H(x_1;\lambda)\, .\end{array}
\end{align}
It follows that the steady states of the normal form satisfy $x_1=x_2=x_3$, where $x_1$ satisfies $H(x_1;\lambda)= C(\lambda) + D(\lambda)x_1 +E(\lambda)x_1^2 + \mathcal{O}(x_1^3)=0$. Under the generic conditions that $C'(0), E(0)\neq 0$, this yields the fully synchronous saddle node branches
$$x_1 = x_2 =x_3=\pm\sqrt{-(C'(0)/E(0))\lambda} + \mathcal{O}(\lambda)\, .$$
Again, one of these branches is stable and the other one is unstable in the direction of the maximal synchrony space.
\item[{\bf 2.}] 
When $a_1=0, a_2\neq 0$ and $a_1+a_2+a_3\neq 0$, then $\ker \mbox{ad}_{f_{0,0}^S}$ is spanned by the elements  
$$X_3\ \mbox{and}\ (X_1-X_3)^{\alpha}(X_2-X_3)^{\beta}\ \mbox{with}\ \alpha+\beta \geq 1\, .$$
This time the action of ${\rm ad}^{\Sigma}_{f_{0,0}^N}$ on ${\rm ker}\, {\rm ad}^{\Sigma}_{f_{0,0}^S}$ is nontrivial. The only terms in the kernel that are not in ${\rm im\ ad}_{f_{0,0}^N}$ are actually those of the form 
$$(X_1-X_3)^{\alpha}, X_2-X_3\ \mbox{and} \ X_3,\, \mbox{with} \ \alpha\geq 1\, .$$
This means that the general normal form of $f$ is
$$\overline f(X_1, X_2, X_3;\lambda)= (X_1-X_3)F(X_1-X_3) + A(\lambda)(X_2-X_3) + B(\lambda) X_3 \, ,$$
where $F(X_1-X_3)= C(\lambda) + D(\lambda)(X_1-X_3)+ \mathcal{O}(X_1-X_3)^2$ and $A(0)=a_2, B(0) = a_1+a_2+a_3, C(0)=0$. This gives the equations of motion
\begin{align}
 \begin{array}{l} \dot x_1 = B(\lambda)x_1\, , \\ \dot x_2 = B(\lambda)x_1 + (x_2-x_1)F(x_2-x_1;\lambda) \, ,\\ \dot x_3 = B(\lambda)x_1 +A(\lambda)(x_2-x_1)  + (x_3-x_1)F(x_3-x_1;\lambda)\, .\end{array}
\end{align}
Under the generic assumption that $C'(0), D(0)\neq 0$, we now find three branches of steady states:
\begin{align}\nonumber
& x_1= x_2 = x_3=0\, , \\ \label{solns}
& x_1= x_2=0,  x_3=-(C'(0)/D(0))\lambda + \mathcal{O}(\lambda^2)\, ,\\ \nonumber
& x_1=0,  x_2=-(C'(0)/D(0))\lambda + \mathcal{O}(\lambda^2),  x_3= \pm \sqrt{(a_2C'(0)/D(0)^2)\lambda} + \mathcal{O}(\lambda)\, .
\end{align}
This means that our normal form equations undergo a very particular synchrony breaking steady state bifurcation that comprises a fully synchronous trivial branch, a partially synchronous transcritical branch and fully nonsynchronous saddle-node branches. The solutions on these branches exchange  stability in a specific way, as for example depicted in Figure \ref{figbif}.
 \end{itemize}
 
 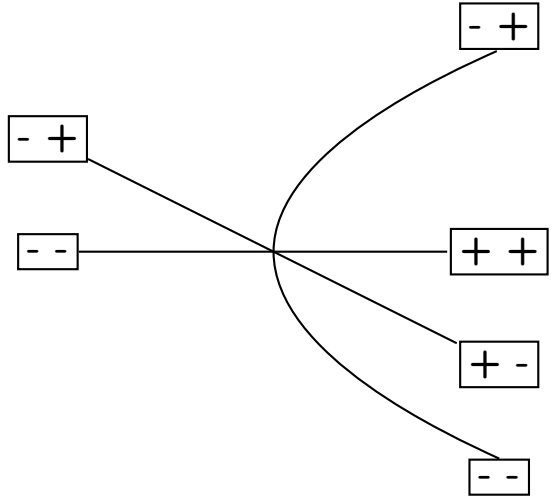
\begin{figure}[ht]\renewcommand{\figurename}{\rm \bf \footnotesize Figure} 
\centering
\begin{tikzpicture} [>=stealth',shorten >=1pt,auto,node distance=3cm,
                    thick,main node/.style={rectangle,draw,font=\sffamily\Large\bfseries}]
 \node[main node] (a) at (-3,0) {-\ -};
 \node[main node] (b) at (3,0) {+\ +};
 \node[main node] (c) at (-3,1.5) {-\ +};
 \node[main node] (d) at (3,-1.5) {+\ -};
 \node[main node] (e) at (3,3) {-\ +};
 \node[main node] (f) at (3,-3) {-\ -};
  \draw (a) --(b);
  \draw (c) --(d);
  \draw[rotate=270] (f.north) parabola bend (0,0) (e.south);
\end{tikzpicture}
\caption{\footnotesize {\rm Bifurcation diagram of a codimension-one steady state bifurcation in the normal form of a three cell feedforward network. Pluses and minuses refer to positive and negative eigenvalues in the eigendirections other than the maximal synchrony space. This figure depicts the solutions of formula (\ref{solns}) in case $a_2, C'(0), D(0)>0$.}}
\label{figbif}
\end{figure}
 
\section{Colored coupled cell networks}\label{generalizationssection}
In this final section, we describe how our results on homogeneous coupled cell networks generalize to certain non-homogeneous coupled cell networks. So let us imagine a coupled cell network with cells of different types. We will refer to the different types of cells as colors. 

More precisely, let us assume that there are $1\leq C <\infty$ colors and that for every color $1\leq c\leq C$ there are precisely $N_c$ cells of color $c$. We label the cells of color $c$ by $1\leq i\leq N_c$ and assume that the state of the $i$-th cell of color $c$ is described by $x_i^{(c)}\in V_c$, where $V_c$ is a linear space that depends on $c$. 

We furthermore assume that the discrete- or continuous-time evolution of $x_i^{(c)}$ is determined by precisely $n_{(1,c)}$ cells of color $1$, by $n_{(2,c)}$ cells of color $2$, etc. This assumption is made precise in Definition \ref{defcolor} below that, although lengthy, is a straightforward generalization of Definition \ref{networkdefinition}. 

\begin{definition}\label{defcolor}
For every $1\leq c, d\leq C$ and every $1\leq j \leq n_{(d,c)}$, assume there is a map 
$$\sigma_j^{(d,c)}:\{1, \ldots, N_c\}\to\{1, \ldots, N_d\}\ .$$
We denote the collection of these maps by 
$$\Sigma:=\{\sigma^{(1,1)}_1, \ldots, \sigma_{n_{(1,1)}}^{(1,1)}; \ldots; \sigma_{1}^{(C,C)}, \ldots, \sigma_{n_{(C,C)}}^{(C,C)}\}\ . $$
Next, we define for all $1 \leq c\leq C$ and $1\leq i\leq N_c$ the maps
\begin{align}
\nonumber
\pi_i^{(c)}: V_1^{N_1} \times&\ldots \times V_C^{N_C}\to V_1^{n_{1,c}}\times \ldots \times V_C^{n_{C,c}}\ \mbox{by} \\ \nonumber 
\pi_i^{(c)}(x^{(1)};\ldots; x^{(C)}) :=& \left(x^{(1)}_{\sigma^{(1,c)}_1(i)}, \ldots, x^{(1)}_{\sigma^{(1,c)}_{n_{(1,c)}}(i)}; \ldots; x^{(C)}_{\sigma^{(C,c)}_1(i)}, \ldots, x^{(C)}_{\sigma^{(C,c)}_{n_{(C,c)}}(i)}\right)\ .
\end{align}
Now assume that $f=(f^{(1)}, \ldots, f^{(C)})$ is a collection of functions, with 
$$f^{(c)}:V_1^{n_{(1,c)}}\times \ldots \times V_C^{n_{(C,c)}}\to V_c\ .$$ 
Then we define
$\gamma_f:V_1^{N_1}\times \ldots \times V_C^{N_C}\to V_1^{N_1}\times \ldots \times V_C^{N_C}$
by 
$$(\gamma_{f})_i^{(c)} := f^{(c)} \circ \pi_i^{(c)} \ \mbox{for all}\ 1\leq c\leq C\ \mbox{and} \ 1\leq i\leq N_c.$$
We say that $\gamma_f$ is a {\it colored coupled cell network map/vector field} subject to $\Sigma$.
\end{definition}

\noindent It is important to note that only the compositions 
$$\sigma_{j_1}^{e,d} \circ \sigma_{j_2}^{d,c}:\{1, \ldots, N_c\}\to\{1,\ldots, N_e\}$$ 
are sensibly defined. This inspires the following definition:
\begin{definition}
We say that $\Sigma$ is a {\it semigroupoid} if for every $1\leq c,d,e\leq C$ and every $1\leq j_1 \leq n_{d,c}$ and $1\leq j_2 \leq n_{e,d}$ there is precisely one $1\leq j_3 \leq n_{e,c} $ such that 
$$\sigma_{j_1}^{(e,d)}\circ\sigma_{j_2}^{(d,c)} = \sigma_{j_3}^{(e,c)}\ .$$ 
\end{definition}
When a collection $\Sigma$ as in Definition \ref{defcolor} is not semigroupoid, then it generates one: the smallest semigroupoid $\Sigma'$ containing $\Sigma$.
\begin{example}
The completely general $C$-dimensional differential equation
\begin{align}
&\dot x^{(c)} = f^{(c)}(x^{(1)}; \ldots; x^{(C)}) \nonumber \ \mbox{for} \ 1\leq c\leq C\ \mbox{and}\ x^{(c)}\in V_c
\end{align}
is an example of a colored coupled cell network with $C$ colors and one cell of each color. The elements of $\Sigma=\{ \sigma_1^{(1,1)}; \ldots; \sigma_C^{(C,C)}\}$ are all defined by $\sigma_j^{(d,c)}(1)=1$. They obviously form a semigroupoid.
\end{example}

\begin{example}
The general $2$-dimensional skew product differential equation
\begin{align}
&\dot x^{(1)} = f^{(1)}(x^{(1)}) \nonumber \\ \nonumber
&\dot x^{(2)}=f^{(2)}(x^{(1)}; x^{(2)})
\end{align}
with $x^{(1)}\in V_1$ and $x^{(2)}\in V_2$ is an example of a colored coupled cell network with two colors and one cell of each color. The elements of $\Sigma=\{ \sigma_1^{(1,1)}; \sigma_1^{(1,2)}, \sigma_1^{(2,2)}\}$ are all defined by $\sigma_1^{(d,c)}(1)=1$ and thus form a semigroupoid. See Figure \ref{pict7}.
 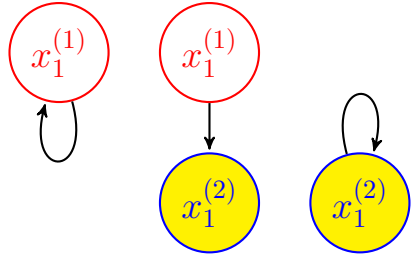
\begin{figure}[ht]\renewcommand{\figurename}{\rm \bf \footnotesize Figure}
\centering
\begin{tikzpicture}[->,>=stealth',shorten >=1pt,auto,node distance=2cm,
                    thick,main node/.style={circle,draw,color=red,font=\sffamily\Large\bfseries},
		    other node/.style={circle,draw,color=blue,fill=yellow,font=\sffamily\Large\bfseries}]

  \node[main node] (1) {$x_1^{(1)}$};
  \node[main node] (2) [right of=1] {$x_1^{(1)}$};
   \node[other node] (3) [below of=2] {$x_1^{(2)}$};

  \node[other node] (4) [right of=3] {$x_1^{(2)}$};

  \path[every node/.style={font=\sffamily\small}]
   (1) edge [loop below] node {} (1)
 (2) edge node {} (3)
(4) edge [loop above] node {} (4)

;
\end{tikzpicture}

\caption{\footnotesize {\rm A colored skew product network.}}
\label{pict7}
\end{figure}

\end{example}

\begin{example}
The $3$-dimensional differential equation
\begin{align}
&\dot x^{(1)}_1 = f^{(1)}(x^{(1)}_2; x^{(2)}_1) \nonumber \\ \nonumber
&\dot x^{(1)}_2 =f^{(1)}(x^{(1)}_2; x^{(2)}_1)\nonumber \\ \nonumber
&\dot x^{(2)}_1 = f^{(2)}(x^{(1)}_2) \nonumber
\end{align}
with $x^{(1)}_1, x^{(1)}_2\in V_1$ and $x^{(2)}_1\in V_2$ is an example of a colored coupled cell network with two colors: two cells of color $1$ and one cell of color $2$. Here, $\Sigma=\{\sigma^{(1,1)}_1, \sigma^{(1,2)}_1; \sigma^{(2,1)}_1\}$ where these maps are defined by 
$$\sigma^{(1,1)}_1(1)=2, \sigma^{(1,1)}_1(2)=2; \sigma^{(1,2)}_1(1)=2; \sigma^{(2,1)}_1(1)=1, \sigma^{(2,1)}_1(2)=1\ .$$
Again, one quickly checks that $\Sigma$ is a semigroupoid. See Figure \ref{pict8}.
 \begin{figure}[ht]\renewcommand{\figurename}{\rm \bf \footnotesize Figure}
\centering
\begin{tikzpicture}[->,>=stealth',shorten >=1pt,auto,node distance=2cm,
                    thick,main node/.style={circle,draw,color=red,font=\sffamily\Large\bfseries},
		    other node/.style={circle,draw,color=blue,fill=yellow,font=\sffamily\Large\bfseries}]

  \node[main node] (1) {$x_1^{(1)}$};
  \node[main node] (2) [below of=1] {$x_2^{(1)}$};
 
   \node[main node] (8) [right of=2] {$x_2^{(1)}$};
  \node[main node] (7) [above of=8] {$x_1^{(1)}$};
 \node[other node] (9) [below of=8] {$x_1^{(2)}$};

  \node[main node] (5) [right of=8] {$x_2^{(1)}$};
   \node[other node] (6) [below of=5] {$x_1^{(2)}$};

  \path[every node/.style={font=\sffamily\small}]
   (2) edge  node {} (1)
 (2) edge [loop below] node {} (2)
(5) edge node {} (6)
 (9) edge node {} (8)
(9) edge [bend right] node  {} (7)
   
;
\end{tikzpicture}
\caption{\footnotesize {\rm An example of a colored network with three cells of two colors.}}
\label{pict8}
\end{figure}
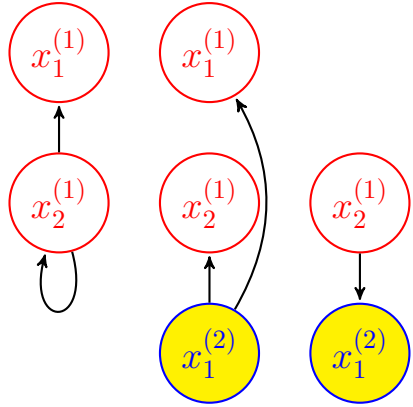

\end{example}

\noindent Under the condition that $\Sigma$ is a semigroupoid, all results of this paper on Lie algebras and normal forms can be generalized to colored coupled cell networks. As an illustration, we state a few facts here without proof.

\begin{theorem}
If $\Sigma$ is a semigroupoid, then for each $\sigma_j^{(d,c)}\in \Sigma$ there is a unique linear map 
$$A_{\sigma_{j}^{(d,c)}}:V_1^{n_{(1,c)}}\times\ldots \times V_C^{n_{(C,c)}} \to V_1^{n_{(1,d)}}\times\ldots \times V_C^{n_{(C,d)}}$$ 
such that for all $1\leq i\leq N_c$ it holds that 
$$A_{\sigma_{j}^{(d,c)}} \circ \pi_i^{(c)}  = \pi^{(d)}_{\sigma_j^{(d,c)}(i)}\ .$$
These maps satisfy the relations $A_{\sigma_{j_1}^{(e,d)}}\circ A_{\sigma_{j_2}^{(d,c)}} = A_{\sigma_{j_1}^{(e,d)}\circ \sigma_{j_2}^{(d,c)}}$ and thus form a representation of the semigroupoid $\Sigma$.
\end{theorem}

\begin{theorem}\label{groupoidbracket}
If $\Sigma$ is a semigroupoid, then 
$$\gamma_f\circ \gamma_g = \gamma_{f\circ_{\Sigma} g}$$
in which $(f\circ_{\Sigma}g)^{(c)}$ is equal to
\begin{align}\nonumber
f^{(c)}\circ \left(g^{(1)}\circ A_{\sigma_1^{(1,c)}}\times  \ldots  \times g^{(1)}\circ A_{\sigma_{n_{(1,c)}}^{(1,c)}}\times \ldots \times g^{(C)}\circ A_{\sigma_1^{(C,c)}}\times \ldots \times g^{(C)}\circ A_{\sigma_{n_{(C,c)}}^{(C,c)}} \right) \ .
\end{align}
\end{theorem}
\begin{theorem}
If $\Sigma$ is a semigroupoid, then
$$[\gamma_f, \gamma_g]=\gamma_{[f,g]_{\Sigma}}$$
in which $[f,g]_{\Sigma}^{(c)}$ equals
$$\sum_{d}\sum_{j} \left( D_{X_j^{(d)}}f^{(c)} \cdot ( g^{(d)}\circ A_{\sigma_j^{(d,c)}} )  -  D_{X_j^{(d)}} g^{(c)} \cdot ( f^{(d)}\circ A_{\sigma_j^{(d,c)}}) \right) \ .$$ 
\end{theorem}
In turn, Theorem \ref{groupoidbracket} can be used to prove normal form theorems for colored coupled cell networks. That is, the theorems of Sections \ref{normalformsection} and \ref{SNsection} remain true with the word ``semigroup'' replaced by ``semigroupoid''. 

We conclude with two results that say that the network symmetries and the robust synchrony spaces of a network remain unchanged by the semigroupoid extension.

\begin{lemma}
Let $\Sigma$ be as in Definition \ref{defcolor}, not necessarily forming a semigroupoid, and let $p$ be a permutation of the cells so that the restriction $p: \{1, \ldots, N_c\}\to \{1, \ldots, N_c\}$ preserves the cells of each color. We say that $p$ is a {\rm network symmetry} if $$p\circ \sigma^{(d,c)}_j=\sigma^{(d,c)}_j\circ p\ \mbox{for all} \ 1\leq c,d\leq C\ \mbox{and all}\ 1\leq j\leq n_{(d,c)}\, .$$
This means that $\lambda_p$ sends orbits of $\gamma_f$ to orbits of $\gamma_f$.

Then the collection of network symmetries of $\Sigma$ is the same as the collection of network symmetries of the semigroupoid $\Sigma'$ generated by $\Sigma$.
\end{lemma}

\begin{lemma}
Let $\Sigma$ be as in Definition \ref{defcolor}, not necessarily forming a semigroupoid, and let $P=\{P^{(1)}, \ldots, P^{(C)}\}$ be a collection of partitions, i.e. for all $1\leq c\leq C$, we have that $P^{(c)}=\{P^{(c)}_1, \ldots, P^{(c)}_{r_c}\}$ is a partition of $\{1, \ldots, N_c\}$. Then the following are equivalent:
\begin{itemize}
\item[i)] The collection of partitions is balanced, i.e. for all $1\leq c,d\leq C$, all $1\leq j\leq n_{(d,c)}$ and all $1\leq k_1 \leq r_c$ there exists a $1\leq k_2\leq r_d$ so that $\sigma^{(d,c)}_j(P^{(c)}_{k_1})\subset P^{(d)}_{k_2}$.
\item[2)]  The subspace 
$$ \hspace{-.5cm} {\rm Syn}_{P}\! :=\! \{x\in V^{N_1}\times\ldots\times V^{N_C}\, |\ x^{(c)}_{i_1}=x^{(c)}_{i_2} \, \mbox{when} \ i_1 \ \mbox{and}\ i_2\ \mbox{are in the same element of}\ P^{(c)} \}$$
is a robust synchrony space for the networks subject to $\Sigma$.
\end{itemize}
The collection of robust synchrony spaces of $\Sigma$ is the same as the collection of robust synchrony spaces of the semigroupoid $\Sigma'$ generated by $\Sigma$.
\end{lemma}

\bibliography{CoupledNetworks}
\bibliographystyle{amsplain}

  \end{document}